\documentclass[12pt]{amsart}
\usepackage[utf8]{inputenc}
\input xypic
\xyoption{all}
\usepackage{fullpage}
\usepackage{url,amsmath}
\usepackage{amssymb}
\usepackage{amsfonts}
\usepackage{amsthm}
\usepackage{mathtools}
\usepackage{amsthm, graphicx, verbatim
}
\usepackage{tikz-cd} 
\usepackage{tikz}
\usetikzlibrary{positioning}

\newcommand{\OO}{\mathcal{O}}
\newcommand{\RR}{\mathbb{R}}

\newcommand{\QQ}{\mathbb{Q}}
\newcommand{\ZZ}{\mathbb{Z}}
\newcommand{\FF}{\mathbb{F}}

\DeclareMathOperator{\card}{card}

\DeclareMathOperator{\Hom}{Hom}
\DeclareMathOperator{\Aut}{Aut}
\DeclareMathOperator{\GL}{GL}

\DeclareMathOperator{\Nrd}{Nrd}
\DeclareMathOperator{\Trd}{Trd}

\DeclareMathOperator{\discrd}{discrd}
\DeclareMathOperator{\disc}{disc}

\DeclareMathOperator{\codiff}{codiff}

\DeclareMathOperator{\Id}{Id}
\DeclareMathOperator{\Ker}{Ker}
\DeclareMathOperator{\End}{End}
\DeclareMathOperator{\lcm}{lcm}

\newtheorem{thm}{Theorem}[section]
\newtheorem{prop}[thm]{Proposition}
\newtheorem{lem}[thm]{Lemma}
\newtheorem{cor}[thm]{Corollary}
\theoremstyle{definition}
\newtheorem{defn}[thm]{Definition}
\newtheorem{rmk}[thm]{Remark}
\newtheorem{alg}[thm]{Algorithm}

\newtheorem*{rmk*}{Remark}
\newtheorem*{ex*}{Example}

\numberwithin{equation}{section}

\author{Kirsten Eisentr\"ager} \author{Gabrielle Scullard}
\address{Kirsten Eisentr\"ager, Department of Mathematics, The
  Pennsylvania State University, University Park, PA 16802, USA, {\tt
    eisentra@math.psu.edu}} \address{Gabrielle Scullard, Department of
  Mathematics, University of Georgia, Athens, GA
  30602, USA, {\tt gabrielle.scullard@uga.edu}} \title[Connecting Kani's Lemma and path-finding in the Bruhat-Tits tree to compute supersingular endomorphism rings]{Connecting Kani's Lemma and path-finding in the Bruhat-Tits tree to compute supersingular endomorphism rings}

\begin{document}
\begin{abstract}
  We give a deterministic algorithm to compute the endomorphism ring of a supersingular
  elliptic curve in characteristic $p$, provided that we are given two noncommuting endomorphisms and the factorization of the discriminant of the ring $\OO_0$ they generate. The algorithm is polynomial in the largest prime factor of the reduced discriminant of $\OO_0$ which is not equal to $p$.
At each prime $q$ for which $\OO_0$ is not maximal, we compute the endomorphism ring locally by computing a $q$-maximal order containing it and, when $q \neq p$, recovering a path to $\End(E) \otimes \ZZ_q$ in the Bruhat-Tits tree. 
We use techniques of higher-dimensional isogenies to navigate towards the local endomorphism ring. 
Our algorithm improves on a previous algorithm which requires a restricted input and runs in subexponential time under certain heuristics.
Page and Wesolowski give a probabilistic polynomial time reduction between computing a single non-scalar endomorphism and computing the endomorphism ring. Beyond using techniques of higher-dimensional isogenies to divide endomorphisms by a scalar, our methods are completely different.
  \end{abstract}
  \maketitle
\section{Introduction}
One of the fundamental problems in computational arithmetic geometry
is computing the endomorphism ring of an elliptic curve.  In the
ordinary case, Bisson and Sutherland \cite{BS2011} gave a
subexponential time algorithm for this problem under certain
heuristics, later improved to rely only on GRH \cite{B12}.  Recently,
Robert outlined an algorithm to compute the endomorphism ring of an
ordinary elliptic curve in polynomial time, assuming access to a
factoring oracle~\cite[Theorem 4.2]{Rob23b}.

In this paper we give a deterministic polynomial time algorithm that
computes the endomorphism ring of a supersingular elliptic curve,
given two non-commuting endomorphisms and a factorization of the
reduced discriminant of the order they generate.  Beside being of
intrinsic interest, computing the endomorphism ring of a supersingular
elliptic curve has become a central problem in isogeny-based
cryptography.  The first cryptographic application of isogenies
between supersingular elliptic curves was the hash function
in~\cite{CGL2009}. The security of this hash function depends on the
hardness of computing endomorphism rings.  More generally, the
hardness of computing endomorphism rings is necessary for the security
of all isogeny-based cryptosystems~\cite{EHLMP18, W21, Wes22}.

Computing the endomorphism ring of a supersingular elliptic curve $E$
was first studied by Kohel~\cite[Theorem 75]{Koh96}, who gave an
approach for generating a subring of finite index of the endomorphism
ring $\End(E)$. The algorithm is based on finding cycles in the
$\ell$-isogeny graph of supersingular elliptic curves in
characteristic $p$ and runs in time $O(p^{1 + \varepsilon})$.  In this paper we complete
Kohel's approach by showing how to compute $\End(E)$ from a suborder.
This builds on \cite{EHLMP20} which gave a 
subexponential algorithm, under certain heuristics, if the input suborder was
Bass.  In~\cite{FIK25}, it is shown that under GRH, a Bass suborder
of $\End(E)$ can be computed in $O(p^{1/2 + \varepsilon})$ time.  In a
different direction, one can attempt to compute $\End(E)$ by
constructing a generating set. One example of a cycle-finding algorithm for computing powersmooth endomorphisms
with complexity $\tilde{O}(p^{1/2})$ and polynomial storage is given
by Delfs and Galbraith~\cite{DG16}.  In~\cite{GPS} it is argued that
heuristically one expects $O(\log p)$ calls to a cycle finding
algorithm until the cycles generate $\End(E)$.  In~\cite{FIK25}, a basis for $\End(E)$ is generated 
using certain inseparable endomorphisms. 
Page and Wesolowski give an
unconditional probabilistic for the
computation of the full endomorphism ring whose complexity is 
$\tilde{O}(p^{1/2})$~\cite{PW23}.

One of the main proposals in isogeny-based cryptography, SIDH, was
broken in July 2022~\cite{CD23,MMPPW2023,Rob2023}. In SIDH, the
underlying hard problem is finding paths in the isogeny graph when
certain torsion-point information is revealed.  The groundbreaking
idea for the break of SIDH relied on moving this problem into a more
flexible framework using isogenies between products of elliptic
curves, via Kani's Lemma.  We apply these techniques to the more
general problem of computing supersingular endomorphism rings. There
are other recent papers which have also exploited higher-dimensional
isogenies, such as the new post-quantum signature scheme
SQIsignHD~\cite{DLRW23}. They are also used to answer the question of how
much knowing one non-scalar endomorphism helps in computing the full
endomorphism ring~\cite{HW25, PW23}. In~\cite{PW23}, it is shown that computing the endomorphism ring of a
supersingular elliptic curve is equivalent under probabilistic
polynomial time reductions to the problem of computing a single
non-scalar endomorphism.

Our algorithm is an improvement and a generalization of the main
result in \cite{EHLMP20}, which we summarize briefly here. Given a
Bass order $\mathcal{O}_0$ and a factorization of the reduced
discriminant $\discrd(\mathcal{O}_0)$, they compute all local maximal
orders containing $\mathcal{O}_0 \otimes \ZZ_q$ at each prime $q$
dividing $\discrd(\mathcal{O}_0)$.  They then combine the local
information to obtain all global maximal orders $\mathcal{O}$
containing $\mathcal{O}_0$ and check if each maximal order $\OO$ is
isomorphic to $\End(E)$.  The Bass restriction on $\mathcal{O}_0$ is
needed to bound the number of maximal orders.

Our algorithm also approaches the problem locally. However, we are
able to determine the local maximal order $\End(E) \otimes \ZZ_q$
without constructing global candidates for $\End(E)$. We also do not
require the Bass restriction, although we can give a more efficient
algorithm when $\mathcal{O}_0$ is Bass.  

There are two key tools.  The
first is a polynomial-time algorithm which, when given an endomorphism
$\beta$ and an integer $n$, determines if $\frac{\beta}{n}$ is an
endomorphism. This algorithm is implicit in Robert's algorithm for
computing an endomorphism ring of an ordinary elliptic
curve~\cite[Section 4]{Rob23b}. A detailed proof and runtime analysis are given
in~\cite[Section 4]{HW25}.
It is also used in \cite{PW23} which gives a probabilistic polynomial time algorithm to compute
the endomorphism ring from an oracle which computes a non-scalar endomorphism.
In contrast, our deterministic algorithm applies 
this algorithm to locate $\End(E) \otimes \ZZ_q$ in the
subtree of the Bruhat-Tits tree that contains a given suborder, 
without needing to enlarge our input order.

The second is a theorem by Tu~\cite[Theorem 8]{Tu11} which expresses an intersection of
finitely many maximal orders in $M_2(\QQ_q)$ as an intersection of at
most three maximal orders, which can be constructed
explicitly. Together with the
containment testing, this allows us to rule out many local maximal orders at once and find $\End(E) \otimes \ZZ_q$.

In the general case, our algorithm works as follows. At each prime $q$
dividing the reduced discriminant, we construct a global order
$\mathcal{O}_q$ containing $\mathcal{O}_0$ which is maximal at
$q$ and is equal to $\mathcal{O}_0 \otimes \ZZ_{q'}$ at all other
primes $q'$. If $q \neq p$, then $\mathcal{O}_q \otimes \ZZ_{q}$
may not be equal to $\End(E) \otimes \ZZ_{q}$, but it is conjugate via
an element of $M_2(\ZZ_q)$ which our algorithm constructs explicitly, step-by-step.
Given the valuation $e = v_q(\discrd(\mathcal{O}_0))$, we are able to recover the correct matrix with at most $4(eq+2)$
applications of an algorithm which determines if an endomorphism is divisible by an integer, summarized in Proposition~\ref{prop:eta} and with more details in Appendix A. In the case that
$\mathcal{O}_0$ is Bass at $q$, in which case there are at most $e+1$
choices, we are able to recover the correct matrix with $4\log_2(e+1)$
applications of Proposition~\ref{prop:eta}. 

\begin{thm}\label{alg:full} There exists an algorithm that computes the endomorphism ring of a supersingular elliptic curve $E$ defined over $\mathbb{F}_{p^2}$ when given $E$, two noncommuting endomorphisms $\alpha_1$ and $\alpha_2$, and a factorization of the reduced discriminant $\Delta$ of the order generated by $\alpha_1$ and $\alpha_2$. We assume $\alpha_1$ and $\alpha_2$ are represented with an efficient HD representation. The algorithm runs in polynomial time in $\log p$, $\log(\deg(\alpha_1))$, $\log(\deg(\alpha_2))$, and $\log \Delta$, and is linear in the number of primes dividing $\Delta/p^{v_p(\Delta)}$ and the largest prime dividing $\Delta/p^{v_p(\Delta)}$. 
\end{thm}

Here, an efficient representation is one which can be represented in polynomial time and for which points can be evaluated efficiently in terms of the size of the input. 

Our algorithm enlarges the order $\OO_0$ generated by $\alpha_1$ and $\alpha_2$ by finding the appropriate local maximal order at each prime $q$ dividing $\Delta/p$. When $q \neq p$, this uses our new results on the intersection of maximal orders in $M_2(\QQ_q)$ to achieve a complexity which is linear in $q$. At the cost of a higher degree polynomial complexity in $q$, one could instead enlarge $\OO_0$ at $q$ by generating linear combinations of basis elements until arriving at an endomorphism in $\OO_0$ which is divisible by $q$, adjoining its division by $q$, and repeating until the order is maximal at $q$. While this algorithm is simpler, our algorithm has lower complexity and introduces completely new techniques. 

This paper is organized as follows. In
Section~\ref{sect:preliminaries}, we give some background on
isogenies and quaternion algebras.
In Section~\ref{sect:btt}, we review the Bruhat-Tits tree and prove results
on intersections of finitely many maximal orders. 
We work with an
order $\OO_0$ that has finite index in the endomorphism ring
$\End(E)$. In Section~\ref{sect:distance}, we take a maximal local
order containing the starting order $\OO_0\otimes \ZZ_q$ and show how to compute
its distance from $\End(E) \otimes \ZZ_q$  in
the Bruhat-Tits tree (Algorithm~\ref{alg:distance}). In Section~\ref{sect:localglobal}, we
give an algorithm to test whether an intersection of maximal
  local orders is contained in
$\End(E) \otimes \ZZ_q$. This is used in Section~\ref{sect:path}, to
find a path from our starting local maximal order to
$\End(E) \otimes \ZZ_q$. In Section~\ref{sect:special}, we give a more efficient algorithm in the Bass case.  Finally in
Section~\ref{sect:fullalg}, we put everything together and give the
algorithm for computing the full endomorphism ring from our starting
order $\OO_0$. We also give an example of our algorithm. 
In Appendix A, we describe the division algorithm and prove its complexity and correctness.
In Appendix B, we describe how to compute an embedding $f: \OO_0 \to M_2(\QQ_q)$.

\section{Preliminaries and Definitions}\label{sect:preliminaries}

\subsection{Orders and Lattices in Quaternion Algebras, Discriminants}

\begin{defn}
  Let $R$ be a domain with field of fractions $F$, and let $B$ be a
  finite-dimensional $F$-algebra.  A subset $M\subseteq B$ is an {\em
    $R$-lattice} if $M$ is finitely generated as an $R$-module
  and $MF=B$.  An {\em $R$-order} $\mathcal{O} \subseteq B$ is an
  $R$-lattice that is also a subring of $B$.
  An order is {\em maximal} if it is not properly contained in another order.
\end{defn}

 \begin{defn} 
Suppose $F$ is a field of characteristic not equal to $2$. An algebra $B$ over $F$ is a \em{quaternion algebra} if there exist $i,j \in B$ such that $1,i,j,ij$ is an $F$-basis for $B$ and $i^2=a$, $j^2 = b$, and $ji=-ij$ for some $a,b \in F^{\times}$. 

When the characteristic of $F$ is $2$, $B$ is a quaternion algebra if there exists an $F$-basis $1,i,j,k$ for $B$ such that $i^2 + i =a$, $j^2=b$, and $k=ij=j(i+1)$ with $a \in F$ and $b \in F^{\times}$. 
\end{defn}

There is a {\em standard involution} on $B$ which maps
$\alpha~=~a_1+a_2i+a_3j+a_4ij$ to
$\overline{\alpha}:= a_1-a_2i-a_3j-a_4ij$.  The {\em reduced trace} of
such an element $\alpha$ is defined as
$\Trd(\alpha) = \alpha + \overline{\alpha}= 2a_1$. The {\em reduced
  norm} is
$\Nrd(\alpha) = \alpha \overline{\alpha}= a_1^2 - aa_2^2 -ba_3^2 +
aba_4^2.$

We say that a quaternion
  algebra $B$ over $\QQ$ {\em ramifies} at a prime $q$ (respectively
  $\infty$) if $B \otimes \QQ_q$ (respectively $B \otimes \RR$) is a
  division algebra.  Otherwise, $B$ is said to be {\em split} at $q$ (respectively $\infty$); in
  this case,  $B \otimes \QQ_q \cong M_2(\QQ_q)$ (respectively, $B \otimes \RR \cong M_2(\RR)$).

  The {\em discriminant} of a quaternion algebra $B$, denoted
  $\disc B$, is the product of the finite primes that ramify in $B$. For an order
  $\mathcal{O} \subset B$ with $\ZZ$-basis
  $\{\beta_1, \beta_2, \beta_3, \beta_4\}$, the {\em discriminant
    of $\mathcal{O}$} is defined to be
  $\disc(\mathcal{O}) \coloneqq |\det(\Trd(\beta_i \beta_j)_{i,j})|
  \in \mathbb{Z}_{>0}$~\cite[p.\ 242]{Voight2021}.  The discriminant of
  an order is always a square. The {\em reduced discriminant}
  $\discrd(\mathcal{O})$ is the positive integer square root of
  $\disc(\OO)$ ~\cite[p.\ 242]{Voight2021}.

\subsection{Endomorphism rings of supersingular elliptic curves}

For a supersingular elliptic curve $E$ defined over a finite field
of characteristic $p$, the endomorphism ring $\End(E)$ is a maximal
order of the quaternion algebra $B_{p,\infty}$, the unique up to isomorphism
quaternion algebra over $\QQ$ ramified at the primes $p$ and $\infty$. For this
paper, computing the endomorphism ring $\End(E)$ means producing a
basis of endomorphisms which can be evaluated at powersmooth torsion
points and generate $\End(E)$.

One key fact that we will need to enlarge the given order $\OO_0$
to $\End(E)$, a maximal order in $B_{p,\infty}$, is the local-global
principle. This states that an order in a global quaternion algebra
is determined by its completions at each prime, see~\cite[Theorem
9.4.9, Lemma 9.5.3]{Voight2021}.
We use the fact that maximality is a
local property, i.e.\ an order $\OO \subseteq B_{p,\infty}$ is maximal
if and only if for all primes $q$, $\mathcal{O}\otimes \ZZ_q$ is a
maximal order. Being maximal can also be expressed in terms of the
reduced discriminant: an order $\mathcal{O}$ in $B_{p,\infty}$ is
maximal if and only if the reduced discriminant $\discrd(\mathcal{O})$
is equal to $p$~\cite[p.\ 375]{Voight2021}. The primes at which
$\mathcal{O}$ fails to be maximal are exactly those primes dividing
$\discrd(\mathcal{O})/p$.

Thus, the local-global principle reduces finding $\End(E)$ to finding
$\End(E) \otimes \ZZ_q$ at each prime $q$ dividing
$\discrd(\mathcal{O})/p$. When $q = p$, there is a unique maximal
order in the division algebra $B_{p,\infty} \otimes \ZZ_p$. In the
case that $q \neq p$, the local order $\End(E) \otimes \ZZ_q$ is a
maximal order of $B_{p,\infty} \otimes \QQ_q \cong M_2(\QQ_q)$.

\begin{defn} Let $\mathcal{O}_0$ be an order in $B_{p,\infty}$. We say that an order
  $\mathcal{O}$ is a {\em $q$-enlargement} of $\mathcal{O}_0$ if
  $\mathcal{O}_0 \subset \mathcal{O}$ and $\mathcal{O} \otimes
  \ZZ_{q'} = \mathcal{O}_0 \otimes \ZZ_{q'}$ for all $q' \neq q$. We
  say that $\mathcal{O}$ is a {\em $q$-maximal $q$-enlargement} if $\mathcal{O}$ is a $q$-enlargement such that $\mathcal{O} \otimes \ZZ_q$ is maximal.
\end{defn}

Let $\mathcal{O} \subset B_{p,\infty}$ be a $\ZZ$-order.
We call $\mathcal{O}$ an {\em Eichler order} if $\mathcal{O} \subseteq
B$ is the intersection of two (not necessarily distinct) maximal orders.
The \textit{codifferent} of an order is
$\codiff(\mathcal{O)} = \{ \alpha \in B: \Trd(\alpha \mathcal{O})
\subseteq \ZZ\}.$ We say
that $\mathcal{O}$ is {\em Gorenstein} if the lattice
codiff$(\mathcal{O})$ is invertible as a lattice
\cite[24.1.1]{Voight2021}. A lattice $I$ is {\em invertible} if there is a lattice $I' \subset B$ such that $II' = O_L(I) = O_{R}(I')$ and $I'I = O_L(I')=O_R(I)$, where $O_L(J)$ denotes the left order of $J$ and $O_R(J)$ denotes the right order of $J$. We call $\mathcal{O}$ {\em Bass}
if every superorder $\mathcal{O}' \supseteq \mathcal{O}$ is
Gorenstein.

\subsection{Representation of endomorphisms}

One of the key tools we will use is an algorithm 
which determines if a rational multiple of an endomorphism is an endomorphism. 
For an efficient running time, the endomorphism must be represented
efficiently. For convenience, we will use the so-called ``HD representation''; others are possible as well.

Following \cite[Section 2.4]{Rob24}, we give the definition of an HD representation.

\begin{defn} Let $E$ be a supersingular elliptic curve and let $\phi$ be an $N$-endomorphism of $E$. The {\em HD representation} of $\phi$ is given by the images $(P_{1,i}, P_{2,i}, \phi(P_{1,i}), \phi(P_{2,i}))$ on the basis $(P_{1,i}, P_{2,i})$ of the $E[\ell_i^{e_i}]$, such that $N':=\prod_{i=1}^r \ell_i^{e_i} > N$. 
\end{defn}

The following theorem summarizes the complexity of the representation. For simplicity, we are stating the theorem in the special case of our setting,
for endomorphisms of supersingular elliptic curves over $\FF_{p^2}$, and stating the complexity for the worst case $8$-dimensional isogeny representation. 

\begin{thm}\cite[Section 5, Theorem 3]{Rob24}~\label{thm:complexity} Let $E$ be supersingular defined over $\FF_{p^2}$ and let $\phi$ be an endomorphism of $E$ of degree $N$, with $N$ prime to $p$, given with HD representation on $E[N']$ for $N'=\prod_{i=1}^{m} \ell_i^{e_i}$. Then $\phi$ can be efficiently embedded into an $8$-dimensional $N'$-isogeny $\Phi$, which can be decomposed as a product of $\ell_i$-isogenies in time $\tilde{O}(m^2 d' e \ell^{8})$ arithmetic operations over $\FF_{p^2}$. Given a point $P \in E(\FF_{q'})$, $\phi(P)$ can be evaluated in $O(me\ell^{8}d'')$ arithmetic operations over $\FF_{p^2}$.

Here, $m$ is the number of primes dividing $N'$, $d'=\max\{\lcm(d_i, d_j)\}$ where $\FF_{p^{2d'}}$ is the field of definition of $E[\ell_i^{e_i}]$, $d''=\max\{\lcm(d_i, d_P)\}$ 
where $\FF_{p^{2d_P}}$ is the field of definition of $P$,
$e=\max\{e_i: 1\leq i\leq r\}$, and $\ell = \max\{\ell_i: 1 \leq i \leq r\}$.  
\end{thm}

In our application, one can show (see the proof of Lemma~\ref{lem:complexity}) that $m$ and $\ell$ can be taken to be $O(\log(N))$ and $e$ can be taken to be $1$.

\begin{rmk} 
	Theorem~\ref{thm:complexity} shows that for a choice of $N'$, an HD representation is an efficient representation in the sense of \cite[Definition 7]{HW25}.
\end{rmk}

\subsection{Division algorithm for endomorphisms}

The following polynomial-time algorithm is crucial to testing whether a local order is contained in the endomorphism ring locally.

\begin{prop}\label{prop:eta} [Divide algorithm]
There exists an algorithm which takes as input an elliptic curve $E$ defined over $\FF_{p^k}$, an endomorphism $\beta \in \End(E)$ in efficient representation, and an integer $n$, and outputs  an efficient representation of $\frac{\beta}{n}$ if $\frac{\beta}{n} \in \End(E)$, and FALSE if $\frac{\beta}{n} \not \in \End(E)$. The algorithm runs in polynomial-time in $\log(p^k)$ and $\log(\deg(\beta))$.
\end{prop}

This algorithm was first outlined by Robert in a special case to
compute endomorphism rings of ordinary curves~\cite[Section
4]{Rob23b}.  The main idea is to use Kani's Lemma to translate the
problem into a higher dimension, where there is enough flexibility to
impose powersmoothness. A proof of Proposition~\ref{prop:eta} is given
in~\cite[Section 4]{HW25}. We provide a proof in
Lemma~\ref{lem:complexity}.

Before~\cite{HW25} was posted, we had written down the details for the algorithm, proof of correctness, and run-time analysis, which we include in~Appendix A.

\section{Local Orders and the Bruhat-Tits Tree}\label{sect:btt}
Given an order $\OO_0$ of finite index in $\End(E)$, we will compute
$\End(E)$ from $\OO_0$ by enlarging it so that locally at a prime $q$
it is maximal and equal to $\End(E) \otimes \ZZ_q$.  When $q=p$ this
step follows from work done in~\cite{Voight13} since
$B_{p,\infty} \otimes \ZZ_p$ is a division algebra and has a unique
maximal order. When $q\neq p$ we will first compute some maximal order
containing $\OO_0 \otimes \ZZ_q$ and then find a path from that
maximal order to $\End(E)\otimes \ZZ_q$.  Here we view both orders as
vertices in the Bruhat-Tits tree for $\GL_2(\QQ_q)$. Our new results
in this section are Lemma~\ref{d3calc}, Lemma~\ref{TuConverse}, and
Corollary~\ref{cor:Lambdatilde}, which expand on work of
Tu~\cite{Tu11} relating the tree of maximal orders containing an order
to the reduced discriminant.

\begin{rmk} Throughout this paper, a path in the Bruhat-Tits tree
  always refers to a {\em nonbacktracking} path.\end{rmk}

\subsection{Bruhat-Tits tree}
For the remainder of this section, fix a prime $q \neq p$. We use the labelling conventions described by Tu~\cite{Tu11}. 

\begin{defn} The {\em Bruhat-Tits tree} is the graph whose vertices are rank 2 $\ZZ_q$-lattices in $(\QQ_q)^2$ up to homothety. Two lattices $L$ and $L'$ are homothetic if there exists $\lambda \in \QQ_{q}$ such that $\lambda L = L'$. Two lattice classes $[L]$ and $[L']$ are connected by an edge if and only if there are representatives $L$ and $L'$ such that $q L' \subsetneq L \subsetneq L'$.

Equivalently (see \cite[Lemma 23.5.2]{Voight2021}), one can consider the vertices of the Bruhat-Tits tree as
maximal orders of $M_2(\QQ_q)$, via the correspondence
$[L] \mapsto \End(L)$. In this case, two maximal orders $\Lambda$ and
$\Lambda'$ are neighbors if and only if
$[\Lambda:\Lambda \cap \Lambda'] = [\Lambda':\Lambda \cap \Lambda'] =
q$.
\end{defn}

Fixing a basis for a lattice $L_0$
and identifying $\End(L_0)$ with $M_2(\ZZ_q)$, we associate to each
basis defining a lattice $L$ a $2\times 2$ matrix $T \in \GL_2(\QQ_q)$ that transforms the basis of $L$
into that of $L_0$.
 
Then $\End(L) = T^{-1} M_2(\ZZ_q) T.$ Given $L_0$ and $L$, the matrix
$T$ is well-defined as an element of
$\QQ_q^*\GL_2(\ZZ_q) \backslash \GL_2(\QQ_q)$. It can be shown (\cite[page 1141]{Tu11}) that standard coset representatives can be taken to be
\begin{equation}
T = \begin{pmatrix}q^a & c\\0 & q^b \\ \end{pmatrix}
\end{equation}
with $a,b \geq 0$, $c \in \ZZ_q$ which can be taken in the set $\{0,1,\ldots, q^{b}-1\}$, and $v_q(c) = 0$ if both $a$ and $b$
are positive. 

The Bruhat-Tits tree is a $(q+1)$-regular tree. Each neighbor of a
lattice $L$ corresponds to a choice of cyclic sublattice of index $q$,
which corresponds to a choice of matrices of the form $\begin{psmallmatrix}
  1 & c\\0 & q \end{psmallmatrix}$ and $\begin{psmallmatrix}q & 0\\0 &
  1 \end{psmallmatrix}$. More generally, a path of length
$n$ starting at the root $M_2(\ZZ_q)$ (labelled by the $2 \times 2$
identity matrix) corresponds to a product of such matrices. We make the following definition.

\begin{defn}\label{def:matrixpath} For each $c$ such that $0 \leq c \leq q-1$, let $\gamma_c
  := \begin{psmallmatrix} 1 & c\\0 & q \end{psmallmatrix}$. Let $\gamma_{\infty}
  =  \begin{psmallmatrix} q & 0\\0 &1\end{psmallmatrix}$.  Let $\Sigma =
  \{\gamma_c: 0 \leq c \leq q-1\} \cup \{ \gamma_{\infty}\}$. We call
  a finite sequence of matrices $\{c_i\}_{i=1}^n$ a {\em matrix path}
  if each $c_i \in \Sigma$ and $c_{i+1} c_i \not \in q M_2(\ZZ_q)$. 
The {\em length} of the
  matrix path $\{c_i\}_{i=1}^n$ is $n$. We call the product $T = c_n c_{n-1} \ldots c_1$
  the {\em associated matrix}.
\end{defn}

There is a bijection between paths of length $n$ in
the Bruhat-Tits tree starting at $M_2(\ZZ_q)$ and matrix paths of
length $n$. The endpoint of the path corresponding to the matrix path
$\{c_i\}_{i=1}^n$ is the order $T^{-1}M_2(\ZZ_q)T$, where $T$ is the associated matrix. 
The vertices of the path are
\[M_2(\ZZ_q), c_1^{-1}M_2(\ZZ_q) c_1,
  c_1^{-1}c_2^{-1}M_2(\ZZ_q)c_2c_1, \ldots, (c_1^{-1}c_2^{-1} \cdots
  c_n^{-1})M_2(\ZZ_q)(c_n \cdots c_2c_1).\]

Depending on the context, we may represent
               vertices in the Bruhat-Tits tree as maximal orders in
               $M_2(\ZZ_q)$, lattices, or $2\times 2$ matrices $T$ as in Figure 1.

\subsection{Distance} We have the usual notion of distance in the Bruhat-Tits tree.

\begin{defn} The {\em distance} between two vertices in the
  Bruhat-Tits tree $v$ and $v'$, denoted $d(v,v')$, is the length of
  the unique path between $v$ and $v'$. We denote the
  distance between $v$ and $v'$ by $d(v, v')$. Here, $v$ and $v'$ may
  be represented by homothety classes of lattices, maximal orders in
  $M_2(\ZZ_q)$, or the matrices associated to the matrix path.
\end{defn}

\begin{defn} Let $\ell$ be a postive integer and $v$ a vertex in the
  Bruhat-Tits tree. The {\em $\ell$-neighborhood of $v$} is the set \[N_\ell(v) := \{v' : d(v', v) \leq \ell\}.\] 
\end{defn}

We also have the analogous notion of distance to a path and neighborhood of a path.

\begin{defn} Let $P$ be the set of vertices along a path in the Bruhat-Tits tree. The {\em distance} between a vertex $v$ and $P$ is  $\min\{d(v, v') \colon v' \in P\}$. The distance between $v$ and $P$ is denoted $d(v, P)$.
\end{defn}

\begin{defn} Let $P$ be the set of vertices along a path in the Bruhat-Tits tree and let $\ell$ be a nonnegative integer. The {\em $\ell$-neighborhood of $P$} is the set \[N_\ell(P) := \{v' : d(v', P) \leq \ell\}.\] 
\end{defn}

\subsection{Distance and matrix labelling}\label{sec:distance}
This section relates
the distance between two vertices in the Bruhat-Tits tree to the
matrix labelling just described. We will also get a bound on the
distance in terms of the reduced discriminant of the intersection of
the two maximal orders.
\vspace{-.7em}
\begin{figure}
  \includegraphics[scale=0.2]{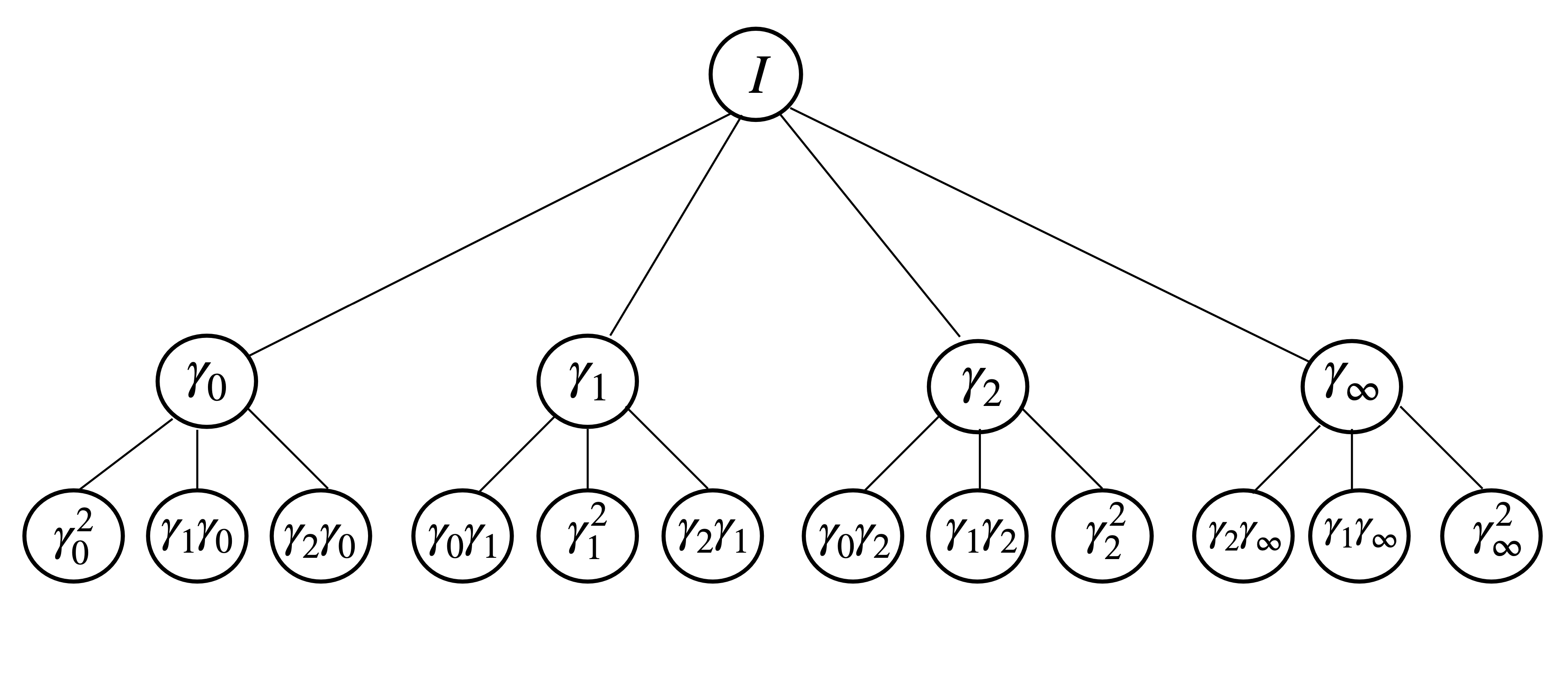}
  \caption{The (truncated) Bruhat-Tits tree for $q=3$, with vertices labelled by the associated matrices. The root of the tree, labelled $I$, corresponds to $M_2(\ZZ_q)$. The vertex labelled with matrix $T$ corresponds to the order $T^{-1}M_2(\ZZ_q)T.$}
\end{figure}

\begin{prop}\label{distmatrices} Let $T = b_r \cdots b_2 b_1 a_k \cdots a_2a_1$ and $T' = c_s \cdots c_{1} a_k \cdots a_2 a_1$, such that $a_i, b_i, c_i \in \Sigma$, the product of two consecutive matrices is not in $M_2(q \ZZ_q)$, and $b_1 \neq c_1$. Then \[d(T^{-1} M_2(\ZZ_q) T, T'^{-1} M_2(\ZZ_q) T') = r+s.\]
\end{prop}

\begin{proof} Let $\gamma = a_k \cdots a_1.$ The unique path from $T^{-1} M_2(\ZZ_q) T$ to $M_2(\ZZ_q)$
  and the unique path from $T'^{-1} M_2(\ZZ_q) T'$
  intersect exactly in the path from $\gamma^{-1} M_2(\ZZ_q) \gamma$
  to $M_2(\ZZ_q).$ Thus, we obtain a (nonbacktracking) path from
  $T^{-1} M_2(\ZZ_q) T$ to $T'^{-1} M_2(\ZZ_q) T'$ by first taking the path
  of length $r$ from $T^{-1} M_2(\ZZ_q) T$ to $\gamma^{-1} M_2(\ZZ_q)
  \gamma$, and concatenating it with the path of length $s$ from $\gamma^{-1} M_2(\ZZ_q) \gamma$ to
  $T'^{-1} M_2(\ZZ_q) T'$.
\end{proof}

We can relate the reduced discriminant of an order to the distance
between maximal orders containing it.

\begin{lem}\label{discdist}
Let $\Lambda$ be a $\ZZ_q$-order in $M_2(\QQ_q)$ such that $\Lambda \subset \Lambda_1 \cap \Lambda_2$ for maximal orders $\Lambda_1, \Lambda_2$. Then $v_q(\discrd(\Lambda)) \geq v_q(\discrd(\Lambda_1 \cap \Lambda_2)).$ 
\end{lem}

\begin{proof} By \cite[Lemma 15.2.15]{Voight2021},
  $\discrd(\Lambda) = [\Lambda_1 \cap \Lambda_2:\Lambda]
  \discrd(\Lambda_1 \cap \Lambda_2).$ The index is a positive integer, so
  $v_q(\discrd(\Lambda)) \geq v_q(\discrd(\Lambda_1 \cap \Lambda_2)).$
\end{proof}

From the lemma, we obtain the two following useful corollaries.

\begin{cor}\label{cor:dist} Suppose $\Lambda \subset M_2(\QQ_q)$ is a $\ZZ_q$-order. Let $e=v_q(\discrd(\Lambda)).$ If
  $\Lambda$ is contained in two maximal orders $\Lambda_1$ and
  $\Lambda_2$, then $d(\Lambda_1,\Lambda_2) \leq e$.
\end{cor}

\begin{proof} When $\Lambda_1$ and $\Lambda_2$ are maximal orders, the
  distance between them is 
  $v_q(\discrd(\Lambda_1 \cap \Lambda_2))$ (\cite[Exercise 23.9]{Voight2021}, see correction \cite[page 11]{VoightErrata}). The result then follows
  from Lemma~\ref{discdist}.
\end{proof}

\begin{cor} Let $\Lambda$ be a $\ZZ_q$-order of finite index in $M_2(\QQ_q)$. Then $\Lambda$ is contained in finitely many maximal orders. 
\end{cor}

\begin{proof} This is immediate from Corollary~\ref{cor:dist}. If $\Lambda_1$ is a maximal order containing $\Lambda$, then all maximal orders containing $\Lambda$ are at most $v_q(\discrd(\Lambda))$ steps from $\Lambda_1$.
\end{proof}

\begin{rmk}
 Suppose $\Lambda \subset \End(E) \otimes \ZZ_q$.
  If we can construct a maximal order $\Lambda_1$ which contains
  $\Lambda$, the preceding corollaries give us a starting point for
  how to locate $\End(E) \otimes \ZZ_q$ in the Bruhat-Tits tree. 
A naive approach would be to check all orders within $e=v_q(\discrd(\Lambda))$ 
steps from $\Lambda_1$ in the Bruhat-Tits tree. 
However, when $e \geq 1$, there are $1 + (q+1) \frac{q^{e} - 1}{q - 1}$ maximal orders at
  most $e$ steps from $\Lambda_1$. Working with each of these orders
  is computationally infeasible for general $\Lambda$. 
  \end{rmk}

\subsection{Finite intersections of maximal orders} 
In this section, we use Tu's results on finite intersections of
maximal orders in $M_2(\QQ_q)$ and our framework of neighborhoods to
describe the set of maximal orders containing such an intersection as the $\ell$-neighborhood
of  a path.
Our main result is Corollary~\ref{cor:Lambdatilde}, which allows us to work with many maximal orders at once.

We give a definition we will use throughout the paper.

\begin{defn}\cite[Notation~7]{Tu11}
Let $S$ be a finite set of maximal orders. We define \[d_3(S) := \max\{d(\Lambda_1, \Lambda_2) + d(\Lambda_2, \Lambda_3) + d(\Lambda_3, \Lambda_1)\},\] where the maximum is taken over all choices of $\Lambda_1, \Lambda_2, \Lambda_3 \in S$. The orders $\Lambda_i$ need not be distinct.
\end{defn}

We restate Tu's main theorem, specialized to our case $K = \QQ_q$.

\begin{thm}\cite[Theorem 8]{Tu11}\label{TuMainTheorem}
Let $S$ be a finite set of maximal orders in $M_2(\QQ_q)$. Let
$\Lambda_1, \Lambda_2, \Lambda_3 \in S$ be such that $d_3(\{\Lambda_1,\Lambda_2, \Lambda_3\}) = d_3(S)$. Then $\bigcap_{\Lambda \in S} \Lambda = \Lambda_1 \cap \Lambda_2 \cap \Lambda_3$. The orders $\Lambda_1, \Lambda_2, \Lambda_3$ need not be distinct.
\end{thm}

Our first lemma relates $d_3(S)$ to the reduced discriminant $\discrd(\bigcap_{\Lambda \in S} \Lambda)$.

\begin{lem}\label{d3calc} Let $S$ be a finite set of maximal orders in $M_2{(\mathbb{Q}_q)}$. Then \[v_q(\discrd(\bigcap_{\Lambda \in S} \Lambda)) = d_3(S)/2.\]
\end{lem}

\begin{proof} By Theorem~\ref{TuMainTheorem}, there are orders $\Lambda_1, \Lambda_2, \Lambda_3 \in S$ such that $\cap_{\Lambda \in S} \Lambda = \Lambda_1 \cap \Lambda_2 \cap \Lambda_3$. There are two cases: that $\Lambda_1, \Lambda_2,$ and $\Lambda_3$ lie along a path in the Bruhat-Tits tree, or they do not.

In the first case, say that $\Lambda_3$ lies on the path between $\Lambda_1$ and $\Lambda_2$. In that case, $d(\Lambda_1, \Lambda_2) = d(\Lambda_1, \Lambda_3) + d(\Lambda_3, \Lambda_2)$, so $d_3(S) = 2 d(\Lambda_1, \Lambda_2)$.  Since the vertices between $\Lambda_1$ and $\Lambda_2$ are the maximal superorders containing $\Lambda_1 \cap \Lambda_2$ (this follows from the argument in \cite[proof of Lemma 12, page 1144]{Tu11}), we have $\Lambda_1 \cap \Lambda_2 \cap \Lambda_3 = \Lambda_1 \cap \Lambda_2$. As $v_q(\discrd(\Lambda_1 \cap \Lambda_2)) = d(\Lambda_1, \Lambda_2)$, the result follows.

In the remaining case, the paths between each pair of the $\Lambda_i$ intersect in a single vertex, $\Lambda_0$. Let $m = d(\Lambda_1, \Lambda_0), n = d(\Lambda_2, \Lambda_0), \ell = d(\Lambda_3, \Lambda_0)$ and assume $m \geq n \geq \ell \geq 0$. Note that since each path passes through $\Lambda_0$, we have $d(\Lambda_1, \Lambda_2) = m+n$, $d(\Lambda_2, \Lambda_3) = n + \ell$, and $d(\Lambda_3, \Lambda_1) = \ell + m$. This shows that $d_3(S) = 2m + 2n+2\ell$.

The intersection $\Lambda_1 \cap \Lambda_2 \cap \Lambda_3$ is conjugate to the order with basis \[ \biggl\{ \begin{pmatrix} 1 & 0\\0 & 1\end{pmatrix}, \begin{pmatrix} 0 & q^n\\0 & 0\end{pmatrix}, \begin{pmatrix} 0 & 0\\q^m & 0\end{pmatrix}, \begin{pmatrix} 0 & 0\\0 & q^\ell\end{pmatrix} \biggl\}.\]  (See \cite[proof of Theorem 2]{Tu11} for details.) As $\disc(\{\alpha_1, \alpha_2, \alpha_3, \alpha_4\}) = |\det(\Trd(\alpha_i \alpha_j))|,$ a computation shows \[\disc(\Lambda_1\cap\Lambda_2\cap\Lambda_3) = q^{2m+2n+2\ell} = q^{d_3(S)}.\] Hence $v_{q}(\discrd(\Lambda_1 \cap \Lambda_2 \cap \Lambda_3)) = d_3(S)/2.$
\end{proof}

We'll also use the following lemma which is key to the proof of Tu's Theorem 8.

\begin{lem}\cite[Lemma 12]{Tu11}\label{Tu} Let
  $S = \{\Lambda_1, \Lambda_2, \Lambda_3\}$ be a set of maximal orders, and let $\Lambda_4$ be a
  maximal order such that $d_3(S \cup \{\Lambda_4\}) = d_3(S).$ Then
  $\Lambda_4 \supset \Lambda_1 \cap \Lambda_2 \cap \Lambda_3.$
\end{lem}

The converse is also true, which we prove in the next Lemma.

\begin{lem}\label{TuConverse} Let $S = \{\Lambda_1, \Lambda_2,
  \Lambda_3\}$ be a set of maximal orders. Suppose $\Lambda_4 \supset \Lambda_1 \cap \Lambda_2 \cap \Lambda_3.$ Then $d_3(S) = d_3(S \cup \{\Lambda_4\}).$
\end{lem}

\begin{proof}  By Lemma~\ref{d3calc}, $v_q(\discrd((\bigcap_{\Lambda \in S} \Lambda ) \cap \Lambda_4)) = d_3(S \cup \{\Lambda_4\})/2,$ and $v_q(\discrd(\bigcap_{\Lambda \in S} \Lambda)) = d_3(S)/2.$ But $\bigcap_{\Lambda \in S} \Lambda \subset \Lambda_4$ implies that $(\bigcap_{\Lambda \in S} \Lambda) \cap \Lambda_4 = \cap_{\Lambda_\in S} \Lambda$, hence the reduced discriminants are equal.  Again by Lemma~\ref{d3calc}, $d_3(S \cup \{\Lambda_4\}) = d_3(S).$
\end{proof}

\begin{cor}\label{cor:Lambdatilde}Let $P$ be the set of maximal orders along a path, and let $\ell \geq 0$. Let $\tilde{\Lambda} = \bigcap_{\Lambda \in N_\ell(P)} \Lambda$. Then the set of maximal orders containing $\tilde{\Lambda}$ is $N_{\ell}(P)$. Moreover, $v_q(\discrd(\tilde{\Lambda})) = 3\ell + \card(P)-1$. \end{cor}

\begin{proof} We will describe $\tilde{\Lambda}$ as an intersection of at most 3 maximal orders.

If $\ell = 0$, then $N_{\ell}(P) = P$. In this case, $\tilde{\Lambda} = \Lambda_1 \cap \Lambda_2$ where $\Lambda_1$ and $\Lambda_2$ are the endpoints of $P$. The only orders containing $\tilde{\Lambda}$ are those in $P$. We have $v_q(\discrd(\Lambda_1\cap \Lambda_2)) = d(\Lambda_1, \Lambda_2) = \card(P) - 1$. 

Now, assume $\ell > 0$. Let $\tilde{\Lambda}_1, \tilde{\Lambda}_2, \tilde{\Lambda}_3$ denote any choice of three orders in $N_{\ell}(P)$, and let $\tilde{\Lambda}_i'$ denote the order on the path $P$ which is closest to $\tilde{\Lambda}_i$.

By the triangle inequality, we have $d(\tilde{\Lambda}_i, \tilde{\Lambda}_j) \leq d(\tilde{\Lambda}_i, \tilde{\Lambda}_i') + d(\tilde{\Lambda}_i', \tilde{\Lambda}_j') + d(\tilde{\Lambda}_j', \tilde{\Lambda}_j) \leq 2 \ell + d(\tilde{\Lambda}_i', \tilde{\Lambda}_j').$ 
Hence $d_3(\{\tilde{\Lambda}_1, \tilde{\Lambda}_2, \tilde{\Lambda}_3\}) \leq 6\ell + d(\tilde{\Lambda}_1', \tilde{\Lambda}_2') + d(\tilde{\Lambda}_2', \tilde{\Lambda}_3') + d(\tilde{\Lambda}_3', \tilde{\Lambda}_1')$.
Since $\tilde{\Lambda}_1', \tilde{\Lambda}_2', \tilde{\Lambda}_3'$ lie along the same path $P$, this sum is at most $2(\card(P)-1)$. 
Hence $d_3(N_{\ell}(P)) \leq 6\ell + 2(\card(P)-1)$. 

Now, choose orders $\Lambda_i \in N_{\ell}(P)$ in the following way:
Choose a path of length $\ell$ starting at an endpoint $\Lambda_1'$ of $P$ which is otherwise disjoint from $P$; the end of this path will be $\Lambda_1$.
To construct $\Lambda_2$, choose a path of length $\ell$ which starts at the opposite endpoint $\Lambda_2'$ of $P$ and is otherwise disjoint from both $P$ and the path from $\Lambda_1$ to $\Lambda_1'$. As the Bruhat-Tits tree is $(q+1)$-regular and $q \geq 2$, this can be done. By construction, $d(\Lambda_1, \Lambda_2) = 2\ell + \card(P)-1$.

We then construct $\Lambda_3$ as follows. Choose a path of length $\ell$ starting at any point of $P$ which is otherwise disjoint from the paths $P$, the path from $\Lambda_1$ to $\Lambda_1'$, and the path from $\Lambda_2$ to $\Lambda_2'$. 
The path from $\Lambda_1$ to $\Lambda_1'$ and $\Lambda_2$ to $\Lambda_2'$ are automatically disjoint unless $P$ is a single point.
Thus, this disjointness restriction can be accomplished if and only if we can choose a path starting at any point of $P$ to avoid two adjacent edges. As the Bruhat-Tits tree is $(q+1)$-regular and $q \geq 2$, we can choose $\Lambda_1, \Lambda_2,$ and $\Lambda_3$ as specified. By construction, $d(\Lambda_1, \Lambda_3) + d(\Lambda_2, \Lambda_3) = d(\Lambda_1, \Lambda_2) + 2\ell = 4\ell + \card(P)-1$.

We summarize the choice of $\Lambda_1, \Lambda_2,$ and $\Lambda_3$ in Figure 2.

\vspace{-.7em}

\begin{figure}[h]
\centering \includegraphics[scale=0.2]{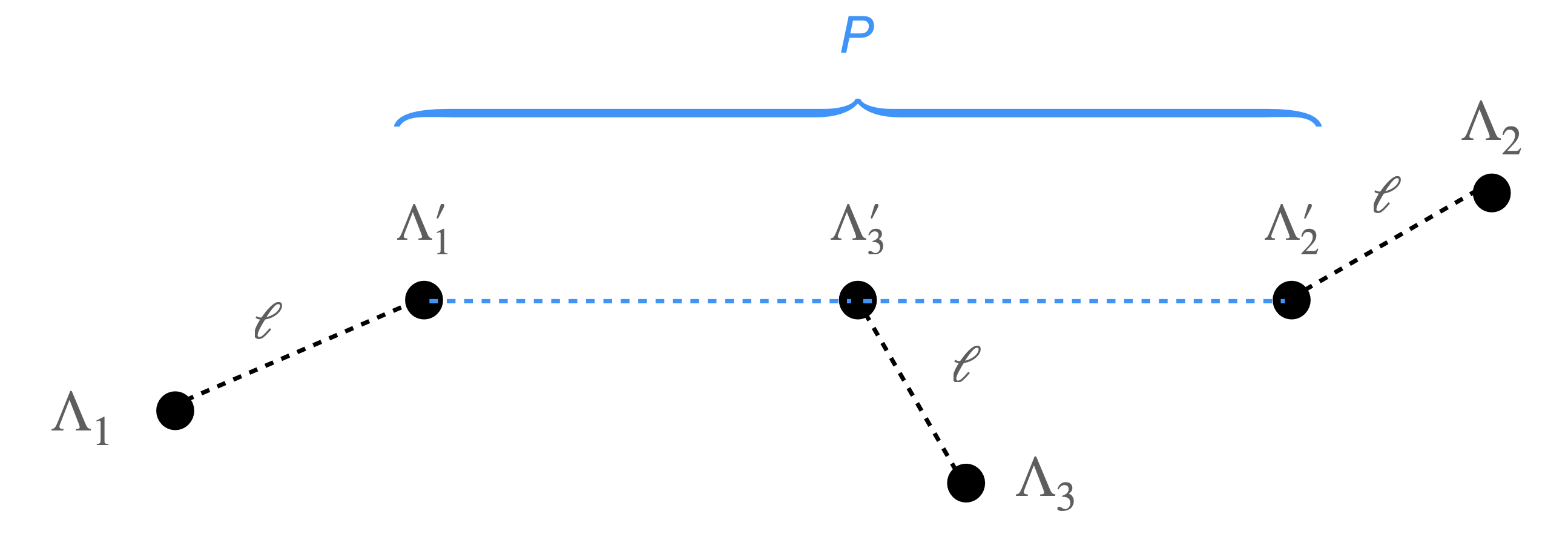} 
\caption{Constructing $\Lambda_1, \Lambda_2, \Lambda_3$ such that $\bigcap_{\Lambda \in N_{\ell}(P)}\Lambda = \Lambda_1 \cap \Lambda_2 \cap \Lambda_3.$}
\end{figure}

\vspace{-.7em}

For this choice of $\Lambda_1, \Lambda_2, \Lambda_3$, we have
$d_3(\{\Lambda_1, \Lambda_2, \Lambda_3\}) = 6\ell +
2(\card(P)-1)$. Thus, $d_3(N_{\ell}(P)) = 6\ell + 2(\card(P)-1)$.  By
Theorem~\ref{TuMainTheorem}, we can write
$\tilde{\Lambda} = \Lambda_1 \cap \Lambda_2 \cap \Lambda_3$, and by
Lemma~\ref{d3calc}, we have
$v_q(\discrd(\tilde{\Lambda})) = 3\ell + \card(P)-1$.

Now, we want to show that the set $S := \{\Lambda \text{ maximal} \colon \tilde{\Lambda} \subset \Lambda\}$ is equal to $N_{\ell}(P)$. It is clear that $N_{\ell}(P) \subset S$ by construction. Choose $\Lambda_4 \not \in N_{\ell}(P)$, so that $d(\Lambda_4, P) > \ell$. We will show that $d_3(\{\Lambda_1, \Lambda_2, \Lambda_3, \Lambda_4\}) > 6\ell + 2(\card(P)-1)$ and hence $\Lambda_4 \not \in S$ by Lemma~\ref{TuConverse}.

Let $\Lambda_4'$ be the point of $P$ which is closest to $\Lambda_4$. Then $d(\Lambda_4, \Lambda_4') > \ell$. By construction, the paths $\Lambda_i$ to $\Lambda_i'$ for $i\leq 3$ are pairwise disjoint except possibly at $\Lambda_i'$. Thus, there is at most one $k$ such that the paths $\Lambda_k$ to $\Lambda_k'$ and $\Lambda_4$ to $\Lambda_4'$ intersect in more than one point.

\textbf{Case 1:} There is no such $k$ or $k=3$. See Figure 3.
 
Consider $d_3(\{\Lambda_1, \Lambda_2, \Lambda_4\})$. For distinct $i,j \in \{1,2,4\}$, the path between $\Lambda_i$ and $\Lambda_j$ passes through $\Lambda_i'$ and $\Lambda_j'$. We have $d(\Lambda_i, \Lambda_j) = d(\Lambda_i, \Lambda_i') + d(\Lambda_i', \Lambda_j') + d(\Lambda_j, \Lambda_j').$ By construction, $d(\Lambda_i, \Lambda_i') = \ell$ if $i \in \{1,2,3\}$, and since $\Lambda_1'$ and $\Lambda_2'$ are the endpoints of $P$, we have $d(\Lambda_1', \Lambda_2') + d(\Lambda_2', \Lambda_4') + d(\Lambda_4', \Lambda_1') = 2(\card(P) - 1)$. Hence  $d_3(\{\Lambda_1, \Lambda_2, \Lambda_4\}) = 4\ell + 2(\card(P)-1) + 2d(\Lambda_4, \Lambda_4') > d_3(S)$ because $d(\Lambda_4, \Lambda_4') > \ell$. Hence $\Lambda_4 \not \in S$.

\textbf{Case 2:} Either $k=1$ or $k=2$. See Figure 3.

Suppose $k=1$.Then
consider $d_3(\{ \Lambda_2, \Lambda_3, \Lambda_4\})$. Arguing as in
the previous case, and noting that $\Lambda_4'$ and $\Lambda_2'$ are
endpoints of $P$, we similarly get  $d_3(\{\Lambda_2, \Lambda_3,
\Lambda_4\}) = 4\ell + 2(\card(P)-1) + 2d(\Lambda_4, \Lambda_4') >
d_3(S)$ because $d(\Lambda_4, \Lambda_4') > \ell$. Hence $\Lambda_4 \not \in S$.  The case $k=2$ is the same argument with $\Lambda_1$ and $\Lambda_2$ switched. \end{proof}

\begin{figure}[h]
\centering
\includegraphics[scale=0.2]{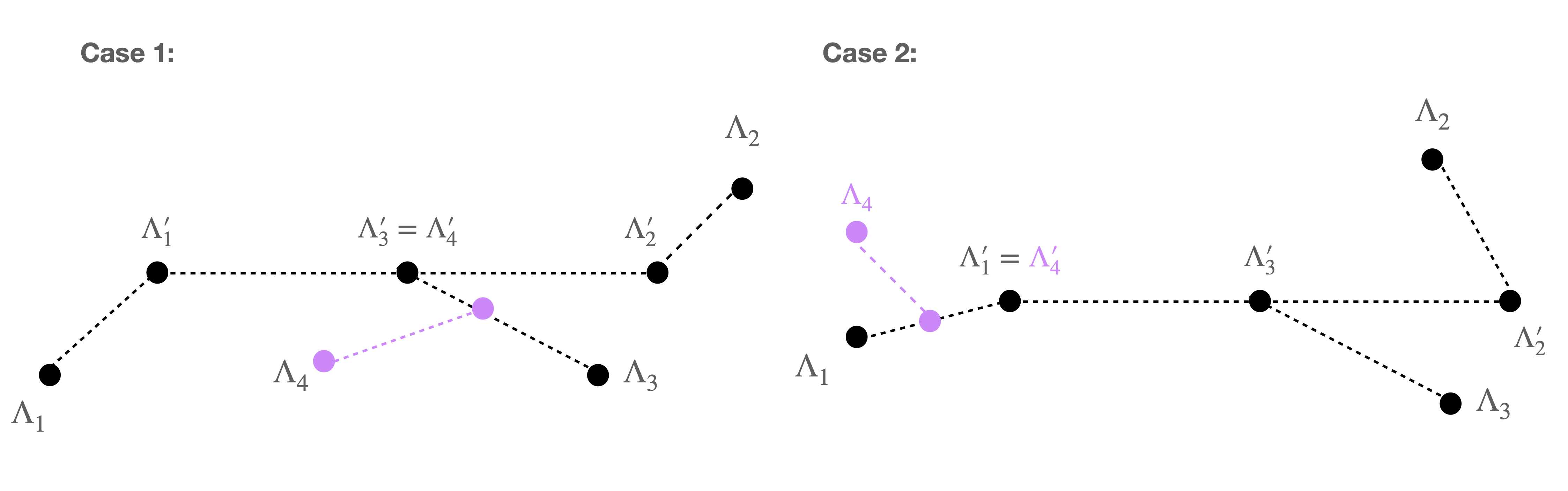}
\caption{Case 1 and Case 2 in the proof of Corollary~\ref{cor:Lambdatilde}.}
\end{figure}

\begin{rmk}
	We note that part of Corollary~\ref{cor:Lambdatilde} can be thought of as a special case of \cite[Proposition 5.3]{A13}
\end{rmk}

\section{Computing the Distance From the Root}\label{sect:distance}

Let $\OO_q$ be a $q$-maximal $q$-enlargement of a suborder of $\End(E)$. In this section, we show how to compute the distance between $\End(E) \otimes \ZZ_q$ and $\OO_q \otimes \ZZ_q$, viewed as vertices on the Bruhat-Tits tree. 
 
First we give a convenient expression for $\bigcap_{\Lambda' \in N_{r}(\Lambda)} \Lambda'$:

\begin{prop}\label{defNr} Fix an integer $r \geq 0$ and a maximal
  order $\Lambda$ in $M_2(\QQ_q)$. Let $\tilde{\Lambda}
  =\cap_{\Lambda' \in N_r(\Lambda)} \Lambda'$. Then $\tilde{\Lambda} =
  \ZZ_q + q^r \Lambda$, 
and the set of maximal orders containing $\tilde{\Lambda}$ is equal to $N_r(\Lambda)$.
\end{prop}

\begin{proof} First, we show that $ \ZZ_q + q^r \Lambda \subset \tilde{\Lambda}$. Equality will follow by showing the reduced discriminants are equal. 

Let $\Lambda' \in N_r(\Lambda)$ and note that this implies $v_q(\discrd(\Lambda \cap \Lambda')) \leq r$. Applying \cite[Lemma 15.2.15]{Voight2021}, we must have $v_q([\Lambda: \Lambda \cap \Lambda']) \leq r$ as well. Therefore, $q^r \Lambda  \subset \Lambda' \cap \Lambda \subset \Lambda'$. Since $\ZZ_q \subset \Lambda'$ for every order $\Lambda'$ in $M_2(\QQ_q)$, this shows that $\ZZ_q + q^r \Lambda \subset \Lambda'$ for all $\Lambda' \in N_r(\Lambda)$. Therefore, $\ZZ_q + q^r \Lambda \subset \tilde{\Lambda}$. 

Now, we show that the reduced discriminant are equal. If $\{1, b_1,b_2, b_3\}$ is a basis for $\Lambda$, then $\{1, q^r b_1, q^r b_2, q^r b_3\}$ is a basis for $\ZZ_q + q^r \Lambda$. By maximality of $\Lambda$, $v_q(\discrd(\Lambda)) = 1$, so by direct computation, $\discrd(\ZZ_q + q^r \Lambda) = q^{3r}$. By Corollary~\ref{cor:Lambdatilde}, $\discrd(\tilde{\Lambda}) = q^{3r}$. This shows that $[\tilde{\Lambda}: \ZZ_q + q^r \Lambda] = 1$ and therefore $\ZZ_q + q^r \Lambda = \tilde{\Lambda}$.

Finally, we show that the set of maximal orders containing $\tilde{\Lambda}$ is equal to $N_r(\Lambda)$. Let $\Lambda_1$ be any order not in $N_r(\Lambda)$. Choose two orders $\Lambda_2$ and $\Lambda_3$ such that $d(\Lambda_2, \Lambda) = d(\Lambda_3, \Lambda) = r$ and the paths from $\Lambda$ to $\Lambda_i$ are disjoint. Since $\Lambda$ is incident to $q+1$ edges, this can be arranged. We compute that $d_3(\{\Lambda_1, \Lambda_2, \Lambda_3\}) = 4r + 2d(\Lambda_1, \Lambda) > d_3(N_r(\Lambda)) = 6r$.  
By Lemma~\ref{TuConverse}, $\Lambda_1 \not \supset \tilde{\Lambda}$. 
\end{proof}

\begin{figure}
\centering
\includegraphics[scale=0.12]{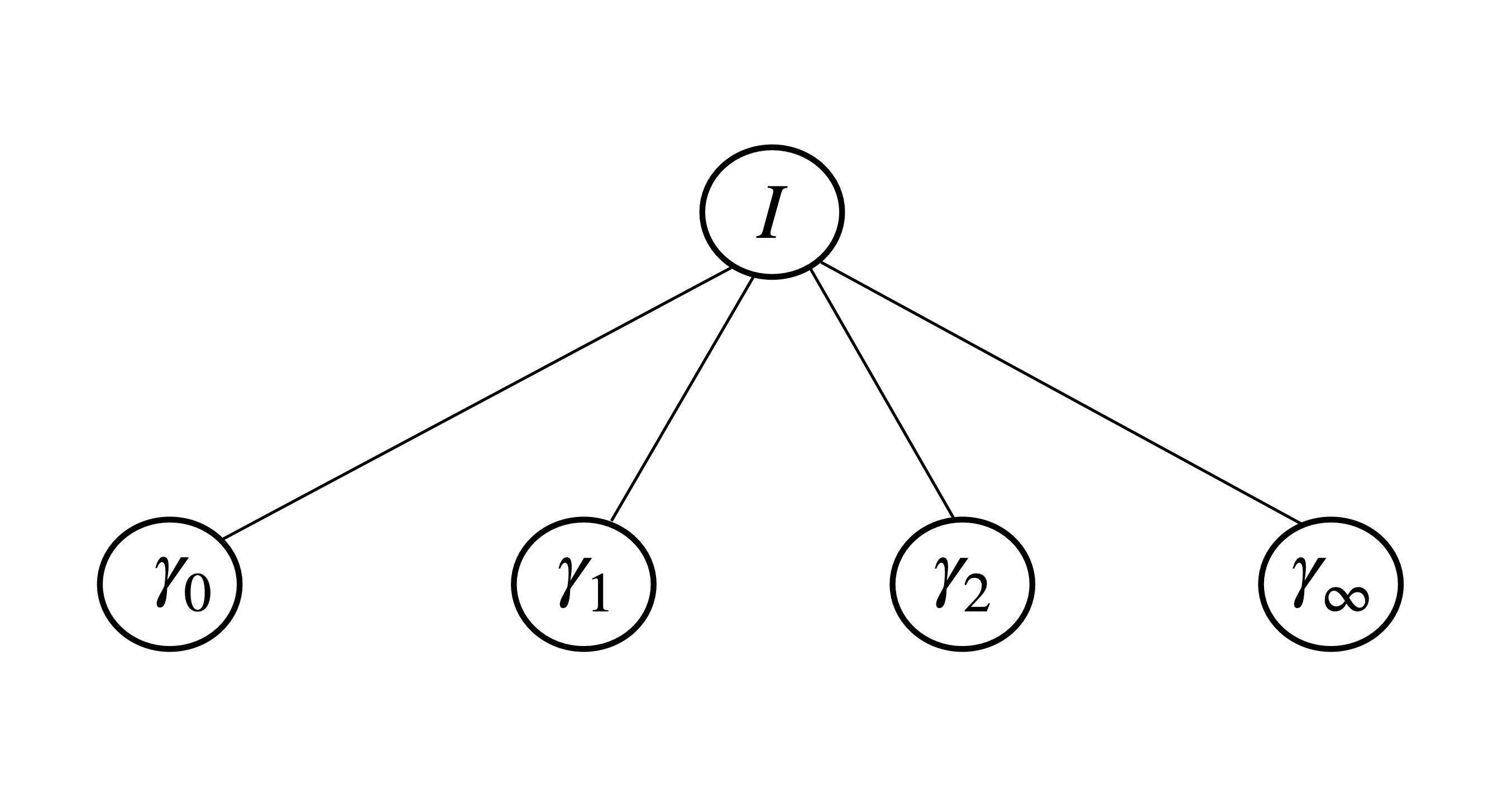}
\caption{The maximal orders containing $\bigcap_{\Lambda' \subset N_{1}(M_2(\ZZ_q))} \Lambda'$ when $q=3$.}

\end{figure}

We will first use Proposition~\ref{defNr} to compute the distance between $\OO_q$ and $\End(E) \otimes \ZZ_q$ under an appropriate embedding into $M_2(\QQ_q)$. Namely,  if $\OO_q \otimes \ZZ_q$ is mapped to the root of the Bruhat-Tits tree, the distance will be the least $r$ for which $\ZZ_q + q^r M_2(\ZZ_q) \subset \End(E) \otimes \ZZ_q$. We compute $r$ by finding the least $r$ for which $q^r \OO_q \subset \End(E)$. We show that such an $r$ exists and is bounded in terms of the reduced discriminant of the input order.

\begin{prop}\label{prop:existsminr} Let $\mathcal{O}_0$ be an order, and let $\mathcal{O}$ be a $q$-enlargement of $\mathcal{O}_0$ for a prime $q$. If $k \geq v_q([\mathcal{O}: \mathcal{O}_0])$, then $q^k \mathcal{O} \subset \mathcal{O}_0$. 
\end{prop}

\begin{proof} For any prime $q'\neq q$, the power of $q'$ exactly dividing the global index $[\OO: \OO_0]$ is a generator for the local index at $q'$, $[\mathcal{O} \otimes \ZZ_{(q')} : \mathcal{O}_0 \otimes \ZZ_{(q')}]$, by \cite[Lemma 9.6.7]{Voight2021}. As $\mathcal{O}$ is a $q$-enlargement of $\mathcal{O}_0$, we have $\mathcal{O} \otimes \ZZ_{q'} = \mathcal{O}_0 \otimes \ZZ_{q'}$, and thus $\mathcal{O}_0 \otimes \ZZ_{(q')} = \mathcal{O} \otimes \ZZ_{(q')}$ \cite[Lemma 9.5.3]{Voight2021}. Hence $[\mathcal{O} \otimes \ZZ_{(q')}: \mathcal{O}_0 \otimes \ZZ_{(q')}]$ is generated by a unit of $\ZZ_{(q')}$ whenever $q' \neq q$. Thus, the global index $[\mathcal{O} : \mathcal{O}_0]$ is  a power of $q$. If $e = v_q([\OO:\OO_0])$, then $q^e \mathcal{O} \subset \mathcal{O}_0$. If $k \geq e$, then $q^k \mathcal{O} \subset q^e \mathcal{O} \subset \mathcal{O}_0$. 
\end{proof}

\begin{cor}\label{cor:discminr} Let $\mathcal{O}_0 \subset \End(E)$ be an order, and let $\OO_q$ be a $q$-maximal $q$-enlargement of $\mathcal{O}_0$ for $q\neq p$. Let $e = v_q(\discrd(\mathcal{O}_0))$. Then $q^e \OO_q \subset \End(E)$.

\end{cor}

\begin{proof} We have $\discrd(\mathcal{O}_0) = [\OO_q : \mathcal{O}_0] \discrd(\OO_q)$ by \cite[Lemma 15.2.15]{Voight2021}. As $\OO_q$ is maximal at $q$, and $q \neq p$, we have $v_q(\discrd(\OO_q)) = 0$. Thus $e = v_q(\discrd(\mathcal{O}_0)) = v_q([\OO_q : \mathcal{O}_0])$. It now follows from Proposition~\ref{prop:existsminr} that $q^e \OO_q \subset \mathcal{O}_0$, and since $\mathcal{O}_0 \subset \End(E)$ by hypothesis, we have $q^e \OO_q \subset \End(E)$.
\end{proof}

The next proposition shows that the distance from $M_2(\ZZ_q)$ can be computed without reference to an embedding into $M_2(\QQ_q)$.

\begin{prop}\label{minrdist} Let $\OO_0$ be a suborder of $\End(E)$ and let $\OO_q$ be a $q$-maximal $q$-enlargement of $\OO_0$.  Let $f$ be any isomorphism $f: \OO_q \otimes \QQ_q \to M_2(\QQ_q)$ such that $f(\OO_q \otimes \ZZ_q) = M_2(\ZZ_q).$  Let $\Lambda_E =  f(\End(E) \otimes \ZZ_q)$. Let $r$ be the least integer such that $q^r \OO_q \subset \End(E)$. Then $r = d(M_2(\ZZ_q), \Lambda_E)$.\end{prop}

\begin{proof}

Let $f$ be any isomorphism satisfying the hypotheses of the proposition, and let $\Lambda_E =  f(\End(E) \otimes \ZZ_q)$. For a nonnegative integer $k$, let $\tilde{\Lambda}_k = \cap_{\Lambda \in N_{k}(M_2(\ZZ_q))} \Lambda$. By Corollary~\ref{cor:Lambdatilde}, $d(M_2(\ZZ_q), \Lambda_E) \leq k$ if and only if $\tilde{\Lambda}_k \subset \Lambda_{E}$. By Proposition~\ref{defNr}, we have $\tilde{\Lambda}_k = \ZZ_q + q^k M_2(\ZZ_q)$. This shows that $d(M_2(\ZZ_q), \Lambda_E)$ is the least $r$ for which $\ZZ_q +q^rM_2(\ZZ_q) \subset \Lambda_E$. We will now show that $\ZZ_q +q^kM_2(\ZZ_q) \subset \Lambda_E$ if and only if $q^k \OO_q \subset \End(E)$. 

As $\ZZ_q$ is contained in every $\ZZ_q$-order, we have $\ZZ_q + q^kM_2(\ZZ_q) \subset \Lambda_E$ if and only if $q^k M_2(\ZZ_q) \subset \Lambda_E$. By hypothesis, $q^k M_2(\ZZ_q) = f(q^k \OO_q \otimes \ZZ_q)$. It follows that $q^k M_2(\ZZ_q) \subset \Lambda_E$ if and only if $q^k \OO_q \subset \End(E) \otimes \ZZ_q$. Since $\OO_q$ is a $q$-enlargement of $\mathcal{O}_0$, at primes $q' \neq q$ we have $q^k \OO_q \otimes \ZZ_{q'} = \OO_q \otimes \ZZ_{q'} = \mathcal{O}_0 \otimes \ZZ_{q'} \subset \End(E) \otimes \ZZ_{q'}.$ It follows from the local-global principle that $q^k \OO_q \subset \End(E)$ if and only if $q^k \OO_q \otimes \ZZ_q \subset \End(E) \otimes \ZZ_q$. This shows that $d(M_2(\ZZ_q), \Lambda_E)$ is the least $r$ for which $q^r \OO_q \subset \End(E)$.\end{proof}

We now give an algorithm to compute the distance $d(M_2(\ZZ_q), \Lambda_E)$, where $\Lambda_E$ is the image of $\End(E) \otimes \ZZ_q$ under any isomorphism satisfying the hypotheses of Proposition~\ref{minrdist}. We will construct both $\OO_q$ and $f$ in Section~\ref{sect:localglobal}.

\begin{alg}\textbf{Computing the distance $d(M_2(\ZZ_q), \Lambda_E)$}\label{alg:distance}

\textbf{Input:} $E/\FF_{p^2}$ supersingular; an order $\mathcal{O}_0 = \langle 1, \alpha_1, \alpha_2, \alpha_1\alpha_2 \rangle$, where $\alpha_1$ and $\alpha_2$ are given in efficient HD representation; a multiplication table for $\mathcal{O}_0$; a basis $B$ for a $q$-maximal $q$-enlargement $\OO_q$ of $\mathcal{O}_0 \subset \End(E)$ with elements of $B$ expressed as $\ZZ[1/q]$-linear combinations of $\{1, \alpha_1, \alpha_2, \alpha_1 \alpha_2\}$; $e := v_q(\discrd(\mathcal{O}_0))$

\textbf{Output:} $r=d(M_2(\ZZ_q), \Lambda_E)$
	\begin{enumerate} 
		\item Set $i:=e-1$. 
		\item While $0 \leq i \leq e-1$: 
			\begin{enumerate}
				\item For each $b \in B$, use Proposition~\ref{prop:eta} with $\beta = q^{i+1} b$ and $n=q$ to determine if $q^i b \in \End(E)$.
				\item If for any $b \in B$, $q^i b \not \in \End(E)$, output $i+1$. Otherwise, set $i:=i-1$.
			\end{enumerate}
		\item Output $0$.
	\end{enumerate}\end{alg}

\begin{prop}\label{prop:algdistance} Algorithm~\ref{alg:distance} is correct and uses at most $4e$ applications of Proposition~\ref{prop:eta}, with input $n=q$ and $\beta \in \End(E)$. 
The run time is polynomial in $\log(p)$ and
$e\log(q)\max\{\log(\Nrd(b)): b \in B\}.$
\end{prop}

\begin{proof} Let $\OO_q$ be the order generated by $B$. By Corollary~\ref{cor:discminr}, we have $q^e \OO_q \subset \End(E)$, so the least $r$ for which $q^r \OO_q \subset \End(E)$ is at most $e$. This shows that the output is the least $r$ for which $q^r \OO_q\subset \End(E)$. By Proposition~\ref{minrdist}, this is the distance $d(M_2(\ZZ_q), \Lambda_E)$. In each iteration of the while loop, we apply Proposition~\ref{prop:eta} at most 4 times, and there are at most $e$ iterations of the while loop. At the $i$-th stage, we can express each candidate endomorphism as $\frac{q^{i+1}b}{q}$, where $q^{i+1}b \in \End(E)$ was verified in the $(i-1)$-th stage. We have $\deg(q^{i+1}b) = q^{2(i+1)}\Nrd(b)$, so the run-time follows from Proposition~\ref{prop:eta}.\end{proof}

\section{Using Global Containment to Test Local Containment}\label{sect:localglobal}

In this section, we show how to translate between computations in
$\End(E) \otimes \QQ$ and computations in the Bruhat-Tits tree.  In
the former, we have Proposition~\ref{prop:eta}, and we would like to
use this algorithm to deduce information about
$\End(E) \otimes \ZZ_q$.  Using work of Voight, we construct a
$q$-maximal $q$-enlargement $\OO_q$ of our input order and an
isomorphism $f: \OO_q \otimes \ZZ_q \to M_2(\ZZ_q)$. This maps the
global order $\OO_q$ to the root of our Bruhat-Tits tree, so that
$\OO_q \subset \End(E)$ if and only if
$M_2(\ZZ_q) = f(\End(E) \otimes \ZZ_q)$.  The main result of this
section is Corollary~\ref{cor:global}, which shows that for any finite
intersection of orders $\Lambda$ in $M_2(\QQ_q)$, we can construct a
global order $\OO$ such that $\OO \subset \End(E)$ if and only if
$\Lambda \subset f(\End(E) \otimes \ZZ_q)$. This will allow us to test
many candidates for $\End(E) \otimes \ZZ_q$ at once.

First, we state the following propositions, which are due
to~\cite{Voight13}. More details of the constructions are given
in~Appendix~B.

\begin{prop}\label{Oq} Suppose an order $\mathcal{O}_0 \subset B_{p,\infty}$ is given by a basis and a multiplication table, and let $q$ be a prime. Then there is an algorithm which computes a $q$-maximal $q$-enlargement $\OO_q$ of $\mathcal{O}_0$.
The run time is polynomial in the size of the basis and multiplication table. The basis elements which are output are of the form $\frac{\beta}{q^k}$, for $\beta \in \mathcal{O}_0$ and $k \leq e = v_q(\discrd(\mathcal{O}_0))$. Furthermore, $\deg(\beta)$ is polynomial in the degrees of basis elements of $\mathcal{O}_0$ and $q$.\end{prop}

\begin{proof} This can be done using \cite[Algorithms 3.12, 7.9, 7.10]{Voight13}.\end{proof}

\begin{prop}\label{iso} Let $q\neq p$. Given a $q$-maximal order $\OO_q \subset B_{p,\infty}$ and a nonnegative integer $r$, there is an algorithm which computes an isomorphism $f: \OO_q \otimes \ZZ_q \to M_2(\ZZ_q)$ modulo $q^{r+1}$. This isomorphism is specified by the inverse image of standard basis elements $i'$ and $j'$ determined mod $q^{r+1}$ in $\OO_q$, such that \[j' \mapsto \begin{pmatrix} 0 & 1\\1 & 0 \end{pmatrix}\] and \[i' \mapsto \begin{pmatrix} 1 & 0\\ 0 & -1 \end{pmatrix} \text{ if } q \neq 2, \begin{pmatrix}0 & 1 \\ 1 & 1 \end{pmatrix} \text{ otherwise}.\] 
The run time is polynomial in $\log(q^r)$ and multiplication table for $\OO_q$. 
In terms of the basis for $\OO_q$, the representatives $i'$ and $j'$ are expressed with coefficients which are determined mod $q^{r+1}$.\end{prop}

\begin{proof} The main steps are to compute a zero divisor mod $q^{r+1}$ and then to apply \cite[Algorithms 4.2, 4.3]{Voight13}.\end{proof}

The other maximal orders that we will work with in the Bruhat-Tits tree are of the form $T^{-1} M_2(\ZZ_q) T$, where $T$ is a matrix associated to a matrix path, as described in Definition~\ref{def:matrixpath}. The next lemma shows that $T$ can be replaced with $T'$ such that $T \equiv T' \pmod{q^{r+1}}$, where $r > v_q(\det(T))$.

\begin{lem}\label{precision} Let $T$ be as in (3.1). $M \in M_2(\ZZ_q)$ and $e > a+b$. Then there is $C \in \GL_2(\ZZ_q)$ such that $T = C T'$. In particular, $T^{-1} M_2(\ZZ_q) T = T'^{-1} M_2(\ZZ_q) T'$.\end{lem}

\begin{proof}
We write $T' = \begin{psmallmatrix} q^a + dq^e & c + fq^{e}\\ gq^e & q^b + hq^e \end{psmallmatrix}$ with $d,f,g,h \in \ZZ_q$. Multiplying on the left by
$\begin{psmallmatrix} (1 + dq^{e-a})^{-1} &0\\ -gq^{e-a}(1 +
  dq^{e-a})^{-1} & 1 \end{psmallmatrix} \in \GL_2(\ZZ_q)$ gives
$\begin{psmallmatrix} q^a & \alpha \\ 0 & q^b\beta \end{psmallmatrix},$ where
$\alpha = (c + fq^e)(1 + dq^{e-a})^{-1}$ and
$\beta = 1 + hq^{e-b} - gq^{e-a-b}\alpha$.

Since $e > a+b$, we have $v_q(\beta) = 0$ and $v_q(\alpha) \geq 0$. We
also have $c = \alpha + q^b(dq^{e-a-b}\alpha - fq^{e-b})$. Multiply on
the left by
$\begin{psmallmatrix} 1\;\; & \beta^{-1}(dq^{e-a-b}\alpha -
  fq^{e-b})\\ 0\;\; & \beta^{-1} \end{psmallmatrix} \in \GL_2(\ZZ_q)$
to arrive at $T$. All operations are invertible over $\ZZ_q$ provided
$e > a+ b$. \end{proof}

To check whether a finite intersection of local maximal orders is contained in $\End(E) \otimes \ZZ_q$, we will check whether the intersection of related maximal orders is contained in $\End(E)$. 
The following lemma allows us to compute the basis of an intersection of $\ZZ$-orders in polynomial time.

\begin{lem}
  Let $L_1, L_2 \subseteq \ZZ^4$ be two $\ZZ$-lattices of full rank,
  specified by a possibly dependent set of generators. A basis for
  each lattice, a lattice basis for the sum $L_1+L_2$ and for the
  intersection $L_1 \cap L_2$ can be computed in polynomial time.
\end{lem}
\begin{proof}
  A basis for each lattice from a set of $m$ generators can be
  computed by writing the generators as the columns of a matrix and
  the compute the Hermite Normal Form (HNF) of the matrix (see
  \cite[p.\ 149]{GM2002} for the definition). The HNF of this
  $4 \times m$ matrix can be computed in time polynomial in $m$, and
  the bit length of the matrix entries, see~\cite{HM1991,
    MW2001}. Computing a basis for the sum of two lattices immediately
  reduces to the problem of computing a basis of a lattice from a set
  of generators. To compute the intersection of two lattices $L_1,L_2$
  each specified by a $4\times 4$ basis in matrix form $B_1, B_2$, we
  first compute $(B_1^T)^{-1}$ and $(B_2^T)^{-1}$. Here $B_i^T$
  denotes the transpose of $B_i$. The matrix $(B_i^T)^{-1}$ is a basis
  for the dual $\hat{L_i}$ of $L_i$~\cite[p.\ 19]{GM2002}. By Cramer's
  rule, the inverse of this $4 \times 4$ matrix can be computed
  efficiently. Since the dual of
  a lattice $L$ consists of all vectors $y$ in $L \otimes \RR$ whose
  real inner product $\langle x,y\rangle$ is an integer for every $x
  \in L$, it follows easily that the dual of $L_1 \cap L_2$ is
  $\hat{L}_1+ \hat{L}_2$, i.e.\ the smallest lattice containing both
  $\hat{L_1}$ and $\hat{L_2}$. So a basis for the intersection is
  obtained by computing a basis for the lattice $\tilde L:=\hat{L}_1$
  and $\hat{L}_2$ and then computing the dual of $\tilde{L}$. By the
  above argument, this can be computed in polynomial time.
\end{proof}
\begin{cor}\label{cor:intersection}
  Let $\OO(i), i=1, 2, 3$ be orders in $B_{p,\infty}$ such
  that $(\disc(\OO_0))\OO_0 \subseteq \mathcal{O}(i)$. A basis for
  $\cap \OO(i)$ in which each basis vector is given as a $\QQ$-linear
  combination of the basis vectors for $\OO_0$ can be computed in
  polynomial time.
  \end{cor}
\begin{proof}
  We can reduce this to 
  matrix computations with $4 \times 4$ integer matrices and use the
  previous lemma. We identify
  our starting global order $\mathcal{O}_0$ with $\ZZ^4$, whose
  Hermite Normal Form is just
  the $4 \times 4$ identity matrix.
  Since $(\disc(\OO_0))\OO_0 \subseteq \mathcal{O}(i)$, we can scale
  the matrices representing the orders $\OO(i)$ by $\disc \OO_0$
  and work with integer matrices.  See \cite[page 73f.]{Coh93} for generalizing the
  computation of the HNF to matrices with bounded rational
  coefficients. By the previous lemma, 
 $\cap \OO(i)$ can be computed in polynomial time.
  \end{proof}

The following is the main result of this section, which shows that we can check local containment by checking global containment.
  
\begin{cor}\label{cor:global}Let $\mathcal{O}_0 \subset \End(E)$ 
be given with basis $\{1, \alpha_1 , \alpha_2, \alpha_1\alpha_2\}$. 
Let $\Lambda$ be a finite intersection
  of maximal orders of $M_2(\QQ_q)$, and let
  $r \leq v_q(\discrd(\OO_0))$ be an integer such that
  $\Lambda \supset \cap_{\Lambda' \in N_r(M_2(\ZZ_q))} \Lambda'$. Let
  $\mathcal{O}_q$ be a q-maximal q-enlargement of
  $\mathcal{O}_0$ and let $f$ be the isomorphism computed in
  Proposition~\ref{iso}. Then there exists a global order
  $\mathcal{O}$ such that $f(\OO \otimes \ZZ_q) = \Lambda$ and such
  that for this $\OO$, $\mathcal{O} \subset \End(E)$ if and only if
  $\Lambda \subset f(\End(E) \otimes \ZZ_q)$. A basis for
  $\mathcal{O}$ with the property that its elements have degree
  polynomial in $\deg(\alpha_1)$, $\deg(\alpha_2),$ and $q^r$, can be
  computed in polynomial time in the size of a multiplication table
  for $\mathcal{O}_0$ and $\log(q)$.\end{cor}

\begin{proof} We will show that $\mathcal{O}$ can be computed such
  that $f(\mathcal{O} \otimes \ZZ_q) = \Lambda$ and
  $\mathcal{O} \otimes \ZZ_{q'} \subset \mathcal{O}_0 \otimes
  \ZZ_{q'}$ for all primes $q' \neq q$.

  Suppose that $\Lambda$ is maximal. Write
  $\Lambda = T^{-1} M_2(\ZZ_q) T$ where $T$ is the matrix associated
  to a matrix path of length at most $r$, written
  $T=\begin{psmallmatrix} q^a & c\\0 & q^b \end{psmallmatrix}$ where
  $a+b \leq r$.

  We construct an element $t \in \OO_q$ such that
  $f(t) \equiv T \pmod{q^{r+1}}$.  Let $i',j' \in \OO_q$
  denote the inverse image of the standard basis elements of
  $M_2(\ZZ_q)$ modulo $q^{r+1}$, as in Proposition~\ref{iso}.  If $q$
  is odd, then with
  $f(i') \equiv \begin{psmallmatrix} 1 & 0\\ 0 & -1 \end{psmallmatrix}
  \pmod{q^{r+1}}$ and
  $f(j') \equiv \begin{psmallmatrix} 0 & 1\\1 & 0 \end{psmallmatrix}
  \pmod{q^{r+1}}$, we can take
  \[t = \frac{q^a + q^b}{2} + \frac{q^a - q^b}{2}i' + \frac{c}{2}j' +
    \frac{c}{2}i'j'.\] As written, $t$ is an element of
  $\OO_q \otimes \QQ$, but we can replace division by 2
  with multiplication by an integer $m \equiv 2^{-1} \pmod{q^{r+1}}$
  to ensure $t \in \OO_q$ and $f(t) \equiv T \pmod{q^{r+1}}$.
  If $q=2$, then with
  $f(i') \equiv \begin{psmallmatrix} 0 & 1\\ 1 & 1 \end{psmallmatrix}
  \pmod{q^{r+1}}$ and
  $f(j') \equiv \begin{psmallmatrix} 0 & 1\\1 & 0 \end{psmallmatrix}
  \pmod{q^{r+1}}$, we can take
  \[t = (q^a + c) + (q^b - q^a)i' + (c - q^b + q^a)j' + (-c)i'j'.\]

  Let $\mathcal{O} = \frac{1}{q^{a+b}}\hat{t} \OO_q t$, where
  $\hat{}$ denotes the dual isogeny.  At $q' \neq q$, we have
  $\mathcal{O} \otimes \ZZ_{q'} \subset \mathcal{O}_0 \otimes \ZZ_{q'}
  \subset \End(E) \otimes \ZZ_{q'}$. This is because
  $\OO_q \otimes \ZZ_{q'} = \mathcal{O}_0 \otimes
  \ZZ_{q'}$, and
  $\hat{t}\OO_q t \subset \OO_q$.  At $q$,
  we have
  $f(\mathcal{O} \otimes \ZZ_q) = T^{-1}M_2(\ZZ_q)T \supset
  f(\mathcal{O}_0 \otimes \ZZ_q)$. By the local-global principle, we
  get $f(\mathcal{O} \otimes \ZZ_q) \subset f(\End(E) \otimes \ZZ_q)$
  if and only if $\mathcal{O} \subset \End(E)$.

In the general case, let $\Lambda_1, \Lambda_2, \Lambda_3$ be maximal orders in $M_2(\QQ_q)$ such that $\Lambda = \cap_{i=1}^3 \Lambda_i$. We can choose three orders $\Lambda_i$ for which this is true by Theorem~\ref{TuMainTheorem}. Let $\mathcal{O}(i)$ denote a global order such that $\mathcal{O}(i) \otimes \ZZ_{q'} \subset \mathcal{O}_0 \otimes \ZZ_{q'}$ for all $q' \neq q$ and $f(\mathcal{O}(i) \otimes \ZZ_{q}) = \Lambda_i$, as constructed in the previous paragraph. 

Let $\mathcal{O} = \cap_{i=1}^3 \mathcal{O}(i)$. 
By construction, $q^{a+b}\mathcal{O}(i) \subset \OO_q$, and by Corollary~\ref{cor:discminr}, we have $q^{v_q(\discrd(\OO_0))} \OO_q \subset \OO_0$. As $a+b \leq r \leq v_q(\discrd(\OO_0))$, we have $q^{2v_q(\discrd(\OO_0))} \OO(i) \subset \OO_0$, and thus $\disc(\OO_0)\OO(i) \subset \OO_0$. 
By Corollary~\ref{cor:intersection}, a basis of $\OO$ can be computed in polynomial time.

Tensoring by $\ZZ_{q'}$ for any prime $q'$ commutes with taking intersections, as $\ZZ_{q'}$ is a flat $\ZZ$-module. Hence $\mathcal{O} \otimes \ZZ_{q'} = \cap_{i=1}^3 \mathcal{O}(i) \otimes \ZZ_{q'} \subset \mathcal{O}_0 \otimes \ZZ_{q'}$ for all $q' \neq q$, and $f(\mathcal{O} \otimes \ZZ_q) = \cap_{i=1}^3 f(\mathcal{O}(i) \otimes \ZZ_q) = \cap_{i=1}^3 \Lambda_i = \Lambda$. 

We have that $\mathcal{O} \otimes \ZZ_{q'} \subset \mathcal{O}_0 \otimes \ZZ_{q'} \subset \End(E) \otimes \ZZ_{q'}$ for all $q' \neq q$, so $\mathcal{O} \subset \End(E)$ if and only if $\mathcal{O} \otimes \ZZ_q \subset \End(E) \otimes \ZZ_q$. As $f$ is an isomorphism, this is equivalent to $\Lambda \subset f(\End(E) \otimes \ZZ_q)$, as desired.\end{proof}

\section{Finding $\Lambda_E$ in the Bruhat-Tits Tree}\label{sect:path}
Let $\Lambda_E = f(\End(E) \otimes \ZZ_q)$ and $r = d(\Lambda_E, M_2(\ZZ_q))$. Once we have computed $r$, we know that $\Lambda_E$ is of the form $T^{-1}M_2(\ZZ_q)T$, where $T$ is a matrix associated to a matrix path of length $r$. In this section, we show how to recover the matrix path one step at a time, which will allow us to compute $\Lambda_E$.

\begin{prop}\label{defLambdagq}  Suppose $\Lambda$ is a maximal order such that $d(M_2(\ZZ_q), \Lambda) = r$. Let $\gamma$ be the matrix associated to a matrix path $\{c_i\}_{i=1}^k$, where $1 \leq
  k\leq r.$ Then $\Lambda \supset \ZZ_q + q^{r-k}\gamma^{-1} M_2(\ZZ_q) \gamma$ if and only if $\Lambda$ corresponds to a matrix path starting with $\{c_i\}_{i=1}^k$.\end{prop}

\begin{proof} 
	If $\Lambda$ corresponds to a matrix path starting with $\{c_i\}_{i=1}^k$, then the path from $M_2(\ZZ_q)$ to $\Lambda$ consists of the path from $M_2(\ZZ_q)$ to $\gamma^{-1} M_2(\ZZ_q) \gamma$ of length $k$, followed by the path from $\gamma^{-1}M_2(\ZZ_q)\gamma$ of length $r-k$.   Thus, $d(\Lambda, \gamma^{-1}M_2(\ZZ_q) \gamma) = r-k$. By Proposition~\ref{defNr}, this implies $\Lambda \supset \ZZ_q + q^{r-k}\gamma^{-1} M_2(\ZZ_q) \gamma$. 
	Conversely, suppose $\Lambda \supset \ZZ_q + q^{r-k}\gamma^{-1} M_2(\ZZ_q) \gamma$. By Proposition~\ref{defNr}, this implies $d(\Lambda, \gamma^{-1}M_2(\ZZ_q) \gamma) \leq r-k$. Consider the path of length $k$ from $M_2(\ZZ_q)$ to $\gamma^{-1}M_2(\ZZ_q)\gamma$, followed by the path from $\gamma^{-1}M_2(\ZZ_q)\gamma$ to $\Lambda$. This forms a path of length at most $r$ from $M_2(\ZZ_q)$ to $\Lambda$. Since $d(M_2(\ZZ_q), \Lambda) = r$, this must be the unique nonbacktracking path from $M_2(\ZZ_q)$ to $\Lambda$, which implies that the matrix path corresponding to $\Lambda$ must start with $\{c_i\}_{i=1}^k$.
\end{proof}

We obtain the following algorithm to recover the matrix path $\{d_i\}_{i=1}^r$ corresponding to $\Lambda_E$. For each $\gamma_c$ as in Definition~\ref{def:matrixpath}, we check if $\ZZ_q + q^{r-1}\gamma_c^{-1}M_2(\ZZ_q)\gamma_c \subset \Lambda_E$. By Proposition~\ref{defLambdagq}, we have  $\ZZ_q + q^{r-1}\gamma_c^{-1}M_2(\ZZ_q)\gamma_c \subset \Lambda_E$ if and only if $d_1 = \gamma_c$, so we recover the first matrix in the path in $q+1$ checks. Once we have recovered the first $k-1$ matrices, we test each of the $q$ possibilities for $d_k$, continuing until we have recovered the full matrix path.

\begin{alg}\textbf{Computing the path from $M_2(\ZZ_q)$ to $\Lambda_E$} \label{alg:path}

\textbf{Input:} An order $\mathcal{O}_0 \subset \End(E)$; a prime $q \neq p$; $q$-maximal $q$-enlargement $\OO_q$ of $\OO_0$ with basis $B$; $r:=d(M_2(\ZZ_q), \Lambda_E)$; an isomorphism $f: \mathcal{O}_0 \otimes \QQ_q \to M_2(\QQ_q)$ such that $f(\OO_q \otimes \ZZ_q) = M_2(\ZZ_q)$ computed mod $q^{r+1}$; for each $\gamma_c \in \Sigma$, an element $t_c \in \OO_q$ such that $f(t_c) \equiv \gamma_c \pmod{q^{r+1}}$; 

\textbf{Output:} $\gamma$ such that $\Lambda_E = \gamma^{-1}M_2(\ZZ_q)\gamma$
	\begin{enumerate} 
		\item Set $k:=1$, $\gamma':=\Id$, $t':=\Id$, $d_{0}:=\Id$. 
		\item While $k \leq r$:
                  \begin{enumerate}
			\item For each $\gamma_c \in \Sigma$ as in Definition~\ref{def:matrixpath} such that $\gamma_c d_{k-1} \not \in q M_2(\ZZ_q)$:
			\begin{enumerate}
				\item Set $\gamma:=\gamma_c \gamma'$, $t:=t_c t'$
				\item Set $B_{\gamma}:= \{q^{r-2k}\hat{t} b t: b \in B\}$
				\item Apply Proposition~\ref{prop:eta} to each $b_{\gamma} \in B_{\gamma}$ to determine if $b_{\gamma} \in \End(E)$. Each $b_{\gamma}$ can be written as $\frac{\beta}{q^{3}}$ where $\beta \in \End(E)$. 
				\item If for all $b_{\gamma} \in B_{\gamma}$, we have $b_{\gamma} \in \End(E)$: Set $\gamma':=\gamma$, $t':=t$, $d_k:=\gamma_c$, $k:=k+1$, and return to Step 2.
			\end{enumerate}
		\end{enumerate}
		\item Output $\gamma$.

	\end{enumerate}

\end{alg}

\begin{figure}[h]
  \centering \includegraphics[scale=0.22]{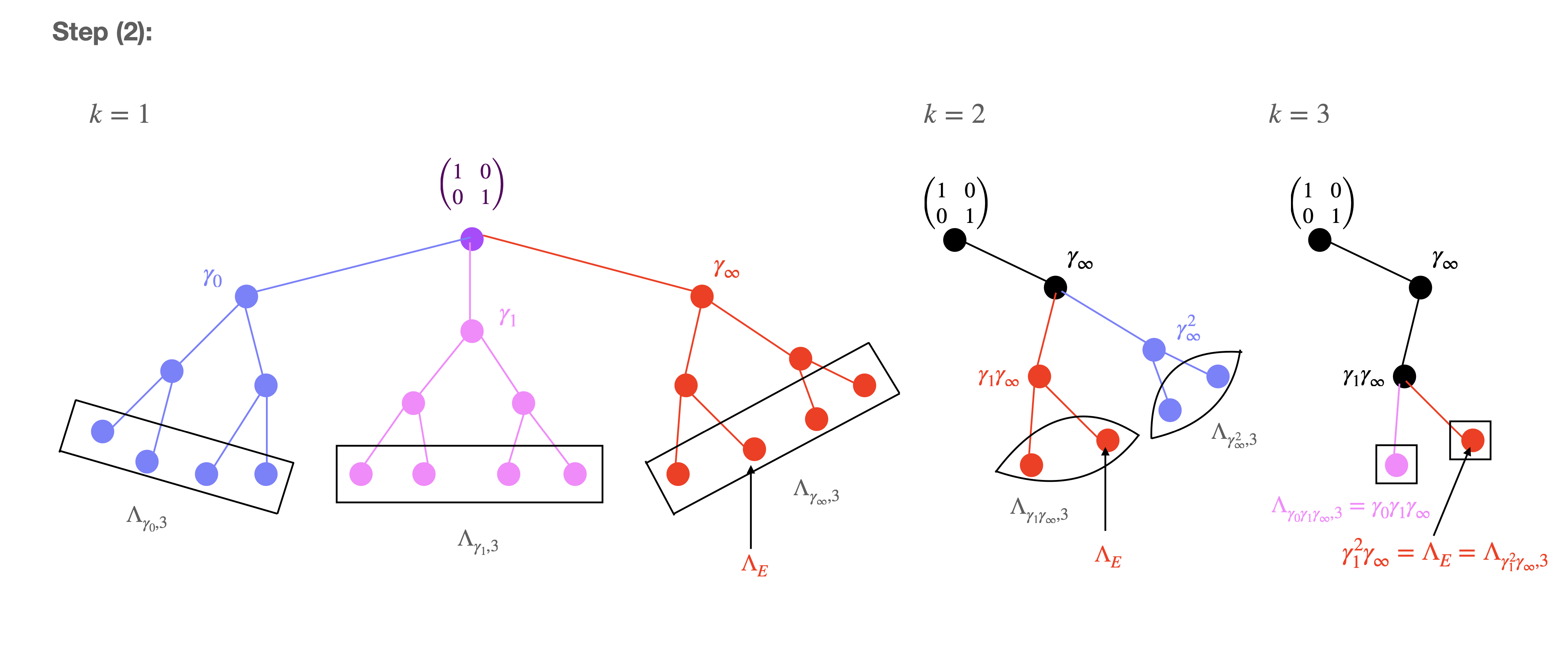}
  \vspace{-3em}
\caption{Algorithm~\ref{alg:path} Step 2 with $q=2$ and $d(\Lambda_E, M_2(\ZZ_q))) = 3$. Black edges indicate the portion of the path determined with previous values of $k$.}
\end{figure}

\vspace{-.5em}

\begin{prop} Algorithm~\ref{alg:path} is correct and requires at most
  $4(rq + 1)$ applications of Proposition~\ref{prop:eta}.
In each
  application, the input for the algorithm in Proposition~\ref{prop:eta} has $n=q^3$.  The inputs for Proposition~\ref{prop:eta} satisfy $\deg(\beta) \leq q^{2r} \max\{\deg(t_c)^{2r}: c=0,1,\ldots, q-1, \infty\}\max\{\deg(b): b\in B\}.$
\end{prop}

\begin{proof}
We know that for some $T_E = d_r d_{r-1} \cdots d_1$, we have $\Lambda_E = T_E^{-1}M_2(\ZZ_q)T_E$. We show that $T_E$ is the output of the algorithm. 

Let $\gamma = c_k \cdots c_1$. In the construction in
Corollary~\ref{cor:global}, with
$\Lambda = \gamma^{-1} M_2(\ZZ_q) \gamma$, we obtain $\OO = \frac{1}{q^{k}} \hat{t} \OO_q t$. 

Letting $\Lambda_{\gamma} = \ZZ_q + q^{r-k} \gamma^{-1} M_2(\ZZ_q) \gamma$, 
the elements of $B_{\gamma}$ generate an order $\OO_{\gamma}$ such that $f(\OO_{\gamma} \otimes \ZZ_{q}) = \Lambda_{\gamma}$ and for $q' \neq q$, $\OO_{\gamma} \otimes \ZZ_{q'} \subset \End(E) \otimes \ZZ_{q'}$.
Thus, $B_\gamma \subset \End(E)$ if and only if
$\Lambda_{\gamma} \subset \Lambda_E$. By
Proposition~\ref{defLambdagq}, we have
$\Lambda_{c_k \cdots c_2c_1} \subset \Lambda_E$ if and only if
$c_i = d_i$ for each $i \leq k$. Therefore, Step 2(a)(iii) has
$B_{\gamma} \subset \End(E)$ if and only if $c_i = d_i$ for each
$i \leq k$. The output after the $k=r$ step must be
$\gamma = d_rd_{r-1}\cdots d_1 = T_E$, as desired.

When $k=1$, there are $q+1$ choices for $\gamma_c$. For $1 < k \leq r$, there are $q$ choices for $\gamma_c$. Thus, we check if $B_{\gamma} \subset \End(E)$ for at most $rq+1$ values of $\gamma$. 

Note that $\Lambda_{\gamma'} \subset \Lambda_{\gamma}$: If $\Lambda'$ is any maximal order such that $d(\Lambda', \gamma M_2(\ZZ_q) \gamma^{-1}) \leq r - k$, then $d(\Lambda', \gamma' M_2(\ZZ_q) (\gamma')^{-1}) \leq r-k+1$ by triangle inequality. By Proposition~\ref{defLambdagq}, as $\Lambda_{\gamma'}$ and $\Lambda_{\gamma}$ are equal to the intersection of orders containing them, this shows that $\Lambda_{\gamma'} \subset \Lambda_{\gamma}$. Furthermore,  $v_q([\Lambda_{\gamma}: \Lambda_{\gamma'}]) =  v_{q}\discrd(\Lambda_{\gamma'}) - v_{q}\discrd(\Lambda_{\gamma}) = 3$. As the previous step verifies that $B_{\gamma'} \subset \End(E)$, this shows that elements of $B_{\gamma}$ can be expressed as $\frac{\beta}{n}$ for $\beta \in \End(E)$ and $n =q^3$. By Corollary~\ref{cor:global}, we have that $\deg(\beta)$ is polynomially-sized in $q^r$ and the degrees of basis elements of $\OO_0$. 
\end{proof}

\section{Bass Orders}\label{sect:special}
Let $\Lambda_0 = f(\mathcal{O}_0 \otimes \ZZ_q)$ and $\Lambda_E = f(\End(E) \otimes \ZZ_q)$.
If $\Lambda_0$ is Bass, the subgraph of maximal orders containing $\Lambda_0$ forms a path which can be recovered efficiently. In this case, we can give a simpler algorithm which performs a binary search along the path to find $\Lambda_E$. 
\begin{alg}\textbf{Finding $\Lambda_E$ When $\Lambda_0$ is Bass}\label{alg:Bass}

\textbf{Input:} An order $\mathcal{O}_0 \subset \End(E)$ which is Bass at $q$, with basis elements expressed in an efficient HD representation; $e = v_q(\discrd(\mathcal{O}_0))$

\textbf{Output:} $\gamma$ such that $\Lambda_E = \gamma^{-1}M_2(\ZZ_q)\gamma$
	\begin{enumerate} 
		\item Compute a $q$-maximal $q$-enlargement $\OO_q \supset \mathcal{O}_0$ and an isomorphism $f: \OO_q \otimes \QQ_q \to M_2(\QQ_q)$ such that $f(\OO_q \otimes \ZZ_q) = M_2(\ZZ_q)$ up to precision $q^{e+1}$. Set $\Lambda_0 := f(\mathcal{O}_0 \otimes \ZZ_q)$.
		\item Compute a list $L$ of matrices $T_i$ associated to matrix paths such that $\Lambda_0 \subset T_i^{-1} M_2(\ZZ_q) T_i$.  Index the matrices $T_i$, starting with $i=1$, such that $T_i$ and $T_{i+1}$ are adjacent in the Bruhat-Tits tree. 
		\item While $|L| > 1$:
		\begin{enumerate}
			\item Set $m:= \lfloor\frac{|L|}{2} \rfloor$, $\Lambda_{\text{start}} := T_1^{-1}M_2(\ZZ_q) T_1$, and $\Lambda_{\text{mid}} := T_{m}^{-1}M_2(\ZZ_q)T_{m}$.
			\item Compute a basis $B$ for an order $\mathcal{O}$ such that $f(\mathcal{O} \otimes \ZZ_q) = \Lambda_{\text{start}} \cap \Lambda_{\text{mid}}$ and $\mathcal{O} \otimes \ZZ_{q'} \subset \OO_0 \otimes \ZZ_{q'}$ for all $q' \neq q$. (Corollary~\ref{cor:global}) 
			\item For each $b \in B$, use Proposition~\ref{prop:eta} to determine if $b \in \End(E)$. 
			\item If $b \in \End(E)$ for all $b$, set $L := \{T_i: 1 \leq i \leq m\}$. Otherwise, set $L:= \{T_i: m < i \leq |L|\}$, and reindex the matrices, by replacing the index $i+m$ with the index $i$ for $1 \leq i \leq |L| - m$. 
		\end{enumerate}
		\item Output the single element of $L$.
	\end{enumerate}

\end{alg}

\begin{prop}\label{prop:algbass} Algorithm~\ref{alg:Bass} is correct, requires at most $4\log_2(e+1)$ applications of Proposition~\ref{prop:eta}, and runs in polynomial time in $\log(q)$ and the size of $\mathcal{O}_0$. 
\end{prop}

\begin{proof} Step 1 can be done in polynomial-time by Propositions~\ref{Oq} and~\ref{iso}. For Step 2, the neighbors of $M_2(\ZZ_q)$ which contain $\Lambda_0$ are the ones corresponding to the lattices which are stable under $\Lambda_0$, which can be computed by computing the common eigenspaces of basis elements of $\Lambda_0$. This is independent of the basis chosen for $\Lambda_0$. Since $\Lambda_0$ is Bass, there will be at most two neighbors containing it. The list $L$ has size at most $e+1$, the orders of $L$ form a path, and $L$ can be computed in polynomial time in $\log(q)$ and the size of $\Lambda_0$ by \cite[Algorithm 4.1, Proposition 4.2]{EHLMP20}. As $\Lambda_E$ contains $\Lambda_0$, the order $\Lambda_E$ must be one of the orders in the list $L$. Each iteration of Step 3 tests which half of the path contains $\Lambda_E$ and discards the other half. After $\log_2(e+1)$ loops, there is only one order remaining in $L$, which must be $\Lambda_E$.
\end{proof}

\section{Computing the Endomorphism Ring}\label{sect:fullalg}
Now we can give the full algorithm to compute the endomorphism ring of
$\End(E)$ on input of two noncommuting endomorphisms, $\alpha_1$ and
$\alpha_2$, which generate a subring $\mathcal{O}_0$ of $\End(E)$. At
every prime $q$ for which $\mathcal{O}_0$ is not maximal, we find the
path from $M_2(\ZZ_q)$ to $\End(E) \otimes \ZZ_q$ in the Bruhat-Tits
tree. We emphasize that the only tool we have to distinguish
$\End(E) \otimes \ZZ_q$ from the other orders of the Bruhat-Tits tree
is the existence of an algorithm which determines if a local order
$\Lambda$ is contained in $\End(E) \otimes \ZZ_q$.

\begin{alg}\label{alg:main}\textbf{Computing the Endomorphism Ring}

  \textbf{Input}: A supersingular elliptic curve $E$ defined over
  $\FF_{p^2}$; a suborder $\mathcal{O}_0$ of $\End(E)$ represented by
  a basis $\{1, \alpha_1, \alpha_2, \alpha_1 \alpha_2\}$, such that
  $\alpha_1$ and $\alpha_2$ can be evaluated efficiently on powersmooth
  torsion points of $E$; a factorization of $\discrd(\mathcal{O}_0)$

\textbf{Output}: A basis for $\End(E)$
    \begin{enumerate}
	\item For each prime $q \mid (\discrd(\mathcal{O}_0)/p)$:
	\begin{enumerate}
		\item Test if $\mathcal{O}_0 \otimes \ZZ_q$ is Bass. If so, use Algorithm~\ref{alg:Bass} to compute $\tilde{\OO}_q:=\End(E) \otimes \ZZ_q$.
		\item Compute a $q$-maximal $q$-enlargement $\OO_q$ of $\mathcal{O}_0$. If $q = p$, set $\tilde{\OO}_q:=\OO_q$. Otherwise, proceed to Step 2. \cite[Algorithm 3.12, 7.9, 7.10]{Voight13}
		\item Compute the distance $r$ between $\OO_q \otimes \ZZ_q$ and $\End(E) \otimes \ZZ_q$, considered as vertices in the Bruhat-Tits tree. [Algorithm~\ref{alg:distance}]
		\item Compute an isomorphism $f: \OO_q \otimes \ZZ_q \to M_2(\ZZ_q)$ given modulo $q^{r+1}$, which extends to an isomorphism $f: \OO_q \otimes \QQ_q \to M_2(\QQ_q)$. [Proposition~\ref{iso}]
		\item Compute the matrix $\gamma$ such that $f(\End(E) \otimes \ZZ_q) = \gamma^{-1} M_2(\ZZ_q) \gamma$.  [Algorithm~\ref{alg:path}]
		\item Compute a basis for a global order $\tilde{\mathcal{O}}_{q}$ such that $f(\tilde{\mathcal{O}}_q \otimes \ZZ_q) = \gamma^{-1}M_2(\ZZ_q)\gamma$ and $\tilde{\mathcal{O}}_q \otimes \ZZ_{q'} \subset \mathcal{O}_0 \otimes \ZZ_{q'}$ for all $q' \neq q$. [Corollary ~\ref{cor:global}]
	\end{enumerate}
	\item Return a basis for $\mathcal{O}_0 + \sum_{q \mid \discrd(\mathcal{O}_0)/p} \tilde{\mathcal{O}}_q$.

\end{enumerate}

\end{alg}

We now prove Theorem~\ref{alg:full}.

\begin{proof} The order $\mathcal{O}_0 \otimes \ZZ_q$ is Bass if and only if $\mathcal{O}_0 \otimes \ZZ_q$ and the radical idealizer $(\mathcal{O}_0 \otimes \ZZ_q)^{\natural}$ are Gorenstein~\cite[Corollary 1.3]{CSV2021}. An order is Gorenstein if and only if the associated ternary quadratic form is primitive \cite[Theorem 24.2.10]{Voight2021}, which can be checked efficiently. 

  For Steps 1b and 1d, we must compute a multiplication table and
  reduced norm form $Q$ for $\mathcal{O}_0$. Coefficients are given by
  the reduced traces of pairwise products of the basis, which can be
  evaluated efficiently using a modified Schoof's Algorithm by
  computing the trace on powersmooth torsion points (see \cite[Theorem
  6.10]{BCEMP}).

Each substep of Step 1 has polynomial runtime. In the worst case, Step 1 requires $\sum_{i=1}^m 4(e_iq_i+2)$ applications of Proposition~\ref{prop:eta}, where $\discrd(\mathcal{O}_0) = \prod_{i=1}^m q_i^{e_i}$.

By construction, $q^{2e}\mathcal{O}_q \subset \OO_0$, and hence $\disc(\OO_0) \OO_q \subset \OO_0$. By Corollary~\ref{cor:intersection}, a basis for the sum $\mathcal{O}_0 + \sum_{q \mid \discrd(\mathcal{O}_0)/p} \mathcal{O}_q$ can be computed in polynomial-time. 
By construction, the output has completion  $\OO_q \otimes
\ZZ_q = \End(E) \otimes \ZZ_q$ at every prime $q \mid \discrd(\OO_0)$
and also still at all other
primes. By the local-global principle, the output is $\End(E)$. 
\end{proof}

We give an example of our algorithm next. Since the necessary higher-dimensional isogeny algorithms have not been implemented yet to enable us to use the division algorithm, our example works with small primes $q$ so that we can check divisibility directly. 

\begin{ex*}
	Let $p=103$ and let $E$ be the elliptic curve with
        $j$-invariant $69$, given by the model $y^2 = x^3 + 37x +
        38$. For this choice of $p$, $B_{p,\infty}$ has a
        $\mathbb{Q}$-basis of the form $\{1,i,j,ij\}$ with $i^2=-1$
        and $j^2=-103$. 
The goal is to compute a maximal order $\mathcal{\tilde O} \subset
B_{p,\infty}$ with $\mathcal{\tilde O} \cong \End(E)$, given a subring
$\mathcal{O}_0$ of finite index.
	We assume that the starting order $\mathcal{O}_0 \subset \mathcal{\tilde
          O} \subset B_{p,\infty}$
       is given by the $\ZZ$-basis
       \[\left\{1, -11095 - \frac{21}{2}i - 11095 j - \frac{7}{2}ij, -49 -
         \frac{49}{2}i - 49j - \frac{49}{2}ij, \frac{107653}{2} +
         \frac{107653}{2}ij\right\}.\] It can be shown that
       $\mathcal{O}_0 \subseteq \End(E)$, and this order has reduced
       discriminant $\Delta = 7^5 \cdot 13^3 \cdot 103$.  Let
       $\beta_1 := -11095 - \frac{21}{2}i - 11095j - \frac{7}{2}k$,
       $\beta_2 := -49 - \frac{49}{2}i - 49j - \frac{49}{2}ij$, and
       $\beta_3 := \frac{107653}{2} + \frac{107653}{2}ij$. We will
       express basis elements as elements of $B_{p,\infty}$ rather
       than as endomorphisms and, when needed, check when an element
       is an endomorphism. We are using an isomorphism with
       $\End(E) \otimes \QQ$ for which we know that $7i$ and $j$ are
       endomorphisms.

	For each $q \mid(\Delta/p)$, we will now compute a $q$-maximal
        $q$-enlargement of $\OO_0$. The order $\OO$ with basis
        $\{1, -i, -\frac{1}{2}i - \frac{1}{2}ij, \frac{1}{2} +
        \frac{1}{2}j \}$ is maximal and contains $\OO_0$. For $q=7$ and
        $q=13$, there is an isomorphism
        $\OO \otimes \ZZ_{q} \to M_2(\ZZ_q)$, given by
        $1 \mapsto \begin{psmallmatrix} 1& 0\\ 0&
          1\end{psmallmatrix}$,
        $i \mapsto \begin{psmallmatrix} 0& -1 \\ 1 & 0\end{psmallmatrix}$,
        $j\mapsto \begin{psmallmatrix} 0 & a\\ a& 0 \end{psmallmatrix}$, and
        $k \mapsto \begin{psmallmatrix} -a &0\\0 & a \end{psmallmatrix}$, where
        $a$ is a $q$-adic square root of $-103$.

	For $q=7$, a 7-maximal 7-enlargement of $\OO_0$ is given by
        the order $\mathcal{O}_7$ with basis
        \[\left\{1, -\frac{1}{2}i - j - \frac{1}{2}ij, -41i - j + 2ij,
          -\frac{1}{2} - \frac{119}{2}i + \frac{9}{2}j -
          \frac{15}{2}ij \right\},\] which has the property that
        $\OO_7 \otimes \ZZ_7 = \OO \otimes \ZZ_{7}$, so we may use the
        same isomorphism for $\OO_7 \otimes \ZZ_{7} \to M_2(\ZZ_7)$.

	We check that $\mathcal{O}_0$ is not Bass at $q=7$. Then, we find the distance between $\OO_7 \otimes \ZZ_7$ and $\Lambda_E$ by finding the least $k$ such that $7^k \OO_7 \subset \End(E)$. 

        In terms of the given basis for $\OO_0$, $\OO_7$ has basis
        given by
        \[\left\{1, 1 + \frac{1}{49} \beta_2, 1 + \frac{43}{7}\beta_1 -
        \frac{47}{49}\beta_2 + \frac{62}{49}\beta_3, -5 +
        \frac{52}{7}\beta_1 - \frac{37}{49}\beta_2 +
        \frac{75}{49}\beta_3\right\}.\] We see that
        $7^2 \OO_7 \subset \OO_0 \subset \End(E)$, so
        $d(M_2(\ZZ_7), \Lambda_E) \leq 2$. We can also check that
        $7 \OO_7 \subset \End(E)$: For example, for
        $\beta = 1 + \frac{1}{49} \beta_2 = -\frac{1}{2}i - j -
        \frac{1}{2}ij$, we have
        \[7 \beta = -\frac{7}{2} i - \frac{7}{2}ij - 7j = -7i\left(
        \frac{1}{2} + \frac{1}{2}j\right) - 7j.\] Since $E$ is
    defined over $\FF_p$ and $j$ is such that $j^2=-103$, $j$
    corresponds to an
        endomorphism of $E$, for example the $p$-power Frobenius on $E$. Since $E[2]$ is
        $\FF_p$-rational, $1+j$ is zero on the 2-torsion of $E$, and hence
        $\frac{1}{2} + \frac{1}{2}j$ is an endomorphism. It can also
        be checked that $7i$ is an endomorphism, so $7\beta \in \End(E)$. For the next basis
        element, we see that
        \[7(1 + \frac{43}{7}\beta_1 - \frac{47}{49}\beta_2 +
        \frac{62}{49}\beta_3) = 22883952 - 287i + 22883945j + 14ij\] is
        an endomorphism because $7i$ and $7ij$ are endomorphisms. For
        the last basis element, we see that
        \begin{align*}7\left( -5 + \frac{52}{7}\beta_1 - \frac{37}{49}\beta_2 +
        \frac{75}{49}\beta_3\right) = -\frac{7}{2} - \frac{833}{2}i +
        \frac{63}{2}j- \frac{105}{2}ij =\\ -7\left(\frac{1}{2} +
        \frac{1}{2}j\right) + 35j - 833i\left(\frac{1}{2} + \frac{1}{2} j\right)
        +728ij \in \End(E).\end{align*} This shows that
        $7 \OO_7 \subset \End(E)$. To see that
        $\OO_7 \not \subset \End(E)$, we can check that the first
        basis element $ -\frac{1}{2}i - j - \frac{1}{2}ij$ is not an
        endomorphism. Since $j$ is an endomorphism, it is enough to
        consider $-\frac{1}{2}i - \frac{1}{2}j$. Since
        $14(-\frac{1}{2}i -\frac{1}{2}ij) = -7i(1 + j)$ is not zero on
        $E[7]$, it follows that $-\frac{1}{2}i - \frac{1}{2}ij$ is not
        an endomorphism, so $\OO_7 \not \subset \End(E)$. Therefore,
        $d(M_2(\ZZ_7), \Lambda_E) = 1$.

	Locally, $\Lambda_E$ must be one of the neighbors of
        $\OO_7 \otimes \ZZ_{7}$. Under the isomorphism of
        $\OO_7 \otimes \ZZ_{7}$ to $M_2(\ZZ_7)$, where the $7$-adic
        square root $-103$ is specified mod $7^2$, we compute bases
        for global orders which are locally equal to neighbors of
        $\OO_7$ and obtain 8 candidates. We check the candidate
        $\frac{\Nrd(\gamma_6)}{7}\cdot (\gamma_6)^{-1}(\OO_7 \otimes
        \ZZ_{7})(\gamma_6)$, whose basis is given
        by \begin{align*}\left\{1, \frac{1051811}{14}i - 69985j -
          \frac{1961725}{14}ij, \frac{39997171}{7}i + 285869j -
          \frac{986434}{7}ij\right.,\\ \left.-\frac{140113}{2} +
          \frac{118328575}{14}i - \frac{2085159}{2}j +
          \frac{8810583}{14}ij\right\}.\end{align*} To check that
        $\frac{1051811}{14}i - 69985j - \frac{1961725}{14}ij$ is an
        endomorphism, for example, we can show that
        $1051811i - 14\cdot 69985j - 1961725ij$ is divisible by
        7. Since $7i$, $j$, and $7ij$ are endomorphisms, we can reduce
        coefficients of $i$ and $ij$ mod 49, and reduced coefficients
        of $j$ mod 7, to check divisibility by 7, to obtain the
        endomorphism $26i - 10ij$. We can check that this is a linear
        combination of $15i - 2ij$ and $17i + ij$, which are
        endomorphisms that are divisible by 7. The other basis
        elements can be handled similarly. This show that
        $\frac{\Nrd(\gamma_6)}{7} \cdot (\gamma_6)^{-1}(\OO \otimes
        \ZZ_{7})(\gamma_6)$ is the endomorphism ring at $q=7$.

Now we consider $q=13$. A $13$-maximal $13$-enlargement of $\OO_0$ is
given by $\OO_{13}$, which has the basis
$\{ 1, -49i, -\frac{1}{2} - \frac{21}{2}i + \frac{7}{2}j -
\frac{7}{2}ij, -\frac{1}{2} - 7i + \frac{21}{2}j + 14ij \}$; since
$\OO_{13} \otimes \ZZ_{13} = \OO \otimes \ZZ_{13}$, we may use the same
isomorphism to $M_2(\ZZ_{13})$.

        At $q=13$, we can check that $\OO_0$ is Bass. In this case, we
        can list the local maximal orders containing
        $\OO_0 \otimes \ZZ_{13}$, and they form a path in the
        Bruhat-Tits tree. As $v_{13}(\discrd(\OO_0)) = 3$, there are
        at most 4 maximal orders orders in this path.

We find that the maximal orders which contain $\OO_0$ locally are $\OO_{13} \otimes \ZZ_{13}$, $\gamma_4^{-1}\OO_{13} \gamma_4 \otimes \ZZ_{13}$,  $\gamma_4^{-1}\gamma_8^{-1} \OO_{13} \gamma_8 \gamma_4 \otimes \ZZ_{13}$, and  $\gamma_4^{-1}\gamma_8^{-1}\gamma_1^{-1} \OO_{13} \gamma_1\gamma_8 \gamma_4 \otimes \ZZ_{13}$. 

Now we can perform a binary search to find the endomorphism ring locally. A basis for $ \OO_{13} \cap \Nrd(\gamma_4)/13\cdot \gamma_4^{-1} \OO_{13} \gamma_4$, where $\gamma_4$ is computed with precision $13^4$, is $\{1, \beta_1, \beta_2, \beta_3\}$, where
\begin{align*}
\beta_1=&-\frac{255783291916664225834427}{2} - 1790483043416725858169326 i\\ &+
\frac{5371449130249643633209815}{2}j + 3580966086832917775040538ij;\\
\beta_2=&-110199150096307809609334 - \frac{3085576202696987671557039}{2}i\\ &+
2314182152022094999300327j + \frac{6171152405391392325644269}{2}ij;
                                                                              \text{
                                                                              and}\\
\beta_3=&-\frac{252873849100292556689371}{2} -
          1770116943702047896825597 i\\ &+
\frac{5310350831106143690476791}{2}j + 3540233887404095793651194 ij.
\end{align*}

We find that the basis elements are contained in $\End(E) \otimes \ZZ_{13}$, since $\frac{1}{2} \pm \frac{1}{2}j$, $7i$, and $7ij$ are endomorphisms, and one can check that the basis elements can be written as $\ZZ$-linear combinations of these endomorphisms. This shows that either $\OO_{13}$ or $\gamma_{4}^{-1} \OO_{13} \gamma_4$  is the endomorphism ring locally. A basis for $\OO_{13}$ is given by $\{1, -49i, -\frac{1}{2} - \frac{21}{2}i + \frac{7}{2}j - \frac{7}{2}ij, -\frac{1}{2} - 7i + \frac{21}{2}j + 14ij \}$, which are all endomorphisms. Therefore, $\OO_{13}$ is the endomorphism ring locally at $q=13$.

Finally, we can compute the endomorphism ring by computing the order
generated by $\OO_0$,
$\Nrd(\gamma_6)/7 \cdot (\gamma_6)^{-1}(\OO_7 \otimes
\ZZ_{7})(\gamma_6)$, and $\OO_{13}$. This gives us the maximal order
$\mathcal{\tilde O} \subseteq B_{p,\infty}$
with basis $ \{1, -\frac{17}{14}i - \frac{1}{14}ij, \frac{15}{7}i - \frac{2}{7}ij,
-\frac{1}{2} - \frac{1}{2}j \}$ and $\mathcal{\tilde O} \cong \End(E)$.

\end{ex*}

\section*{Acknowledgements} We thank Damien Robert for helpful
conversations about the algorithm in~\cite[Section 4]{Rob23b}.
We also
thank Yuri Zarhin for fruitful discussions about automorphisms of
products of elliptic curves.
We thank anonymous reviewers for feedback on a previous draft.

\bibliography{endoring.bib}
\bibliographystyle{alpha}

\newpage
\appendix
\section{Using Higher-Dimensional Isogenies for Testing Divisibility}

In this section, we give background on isogenies between products of
elliptic curves, give a detailed version of the algorithm
in~\cite[Section 4]{Rob23b}, prove correctness of the algorithm, and
give a complexity analysis.

\subsection{Isogenies between polarized abelian varieties and their degrees}\label{sec:notation}
\begin{defn}\cite[p.\ 126]{MilneAv}
  A {\em polarization} of an abelian variety $X$ defined over a field
  $k$ is an isogeny $\lambda: X \to X^{\vee}$ to the dual variety
  $X^{\vee}$ so that $\lambda_{\overline{k}}= \phi_{\mathcal{L}}$ for
  some ample invertible sheaf $\mathcal{L}$ on
  $X_{\overline{k}}$. Here $\phi_{\mathcal{L}}: A(k) \to \operatorname{Pic}(A)$ is
  the map given by $a \mapsto t_a^* \mathcal{L} \otimes
  \mathcal{L}^{-1}$ with $t_a$ the translation-by-$a$ map.
\end{defn}

\noindent{\bf Notation:} Given an isogeny $\Phi: A\to B$ between 
 abelian varieties we denote by $\Phi^{\vee}: B^{\vee} \to A^{\vee}$
 the dual isogeny (see~\cite[p.\ 143]{MumfordAV70}).
\begin{defn}\label{def:N-isogeny}
  Given a positive integer $N$, an {\em $N$-isogeny}
  $\Phi: (A, \lambda_A) \to (B,\lambda_B)$ between principally
  polarized abelian varieties $(A, \lambda_A)$ and $(B, \lambda_B)$ is
  an isogeny such that
  $\Phi^{\vee} \circ \lambda_B \circ \Phi = N \lambda_A$. An
  $(N,N)$-isogeny $\Phi: (A, \lambda_A) \to (B,\lambda_B)$ of abelian
  varieties of dimension $g$ is an $N$-isogeny whose kernel is
  isomorphic to $(\ZZ/N\ZZ)^g$.
\end{defn}

 Let $A$ be an abelian variety with a polarization $\lambda$. Since
  $\lambda$ is an isogeny $A \to \hat A$, it has an inverse in
  $\Hom(\hat A,A) \otimes \mathbb{Q}$. The {\em Rosati involution} on
 $\Hom(\hat A, A) \otimes \mathbb{Q}$ corresponding to $\lambda$ is
 \[ a \mapsto a^{\dagger} = \lambda^{-1} \circ \hat \alpha \circ \lambda.
   \]
   In this paper we will consider endomorphisms of products of
   elliptic curves and abelian varieties. Given an abelian variety
   $A$, an integer $r>1$ and isogenies $\phi_{i,j}: A \to A$ for
   $1\leq i,j\leq r$, the $r \times r$ matrix $M =(\phi_{i,j})_{1\leq
     i,j \leq r}$
   represents the isogeny \begin{align*}\Phi:\, & A^r \to A^r {\text{
         sending }}\\&(P_1, \dots,P_r) \text{ to } \left(\phi_{1,1}(P_1) +
     \dots+\phi_{1,r}(P_r), \dots,\phi_{r,1}(P_1) + \dots+
     \phi_{r,r}(P_r)\right).\end{align*}

   We refer to this as the {\em matrix form of $\Phi$}.

\begin{defn}
  Let $A$ be a principally polarized abelian variety. Consider
  $\Phi:A^r \to A^r$ with matrix form $M= (\phi_{i,j})_{1 \leq
    i,j \leq r}$ as above.
 Let $\phi_{i,j}^{\dagger}: A \to A$ be the Rosati involution of
$\phi_{i,j}$.
Define $\hat \Phi: A^r \to A^r$ as the endomorphism represented by the matrix
$\hat M=(\phi_{j,i}^{\dagger})_{1 \leq i,j \leq r}$.
\end{defn}
Definition~\ref{def:N-isogeny} can also be rephrased as follows,
see~\cite[Section 3.1]{Rob2023}.
\begin{prop}
Let $A$ be principally polarized, and let $\Phi: A^r\to A^r$ be an
isogeny with matrix form $M$. Then $\hat M \cdot M = N\cdot Id_r$ if and only if $\Phi$ is an $N$-isogeny with respect to the product polarization.\end{prop}

\begin{prop}\label{degreebound}  Let $E$ be an elliptic curve. Let $\Phi: E^r
  \to E^r$ be an $N$-isogeny of principally-polarized abelian
  varieties whose matrix form is $M= (\phi_{i,j})_{1 \leq i,j \leq r}$. Then the degrees of the isogenies $\phi_{i,j}: E \to E$ are
  bounded above by $N.$
\end{prop}

\begin{proof} By the previous proposition,
  $\hat M \cdot M = N \cdot \Id_r$. In particular, the $i$-th diagonal
  entry of $\hat M \cdot M$ is given by
  $\sum_{j=1}^r \phi_{j,i}^{\dagger}\phi_{j,i} = N$. For elliptic
  curves, $\phi_{j,i}^{\dagger}$ is the dual isogeny of $\phi_{j,i}$, so we
  have $\sum_{j=1}^r \deg(\phi_{j,i}) = N$ (by convention, the degree
  of the 0 map is 0). As the degree of an isogeny is nonnegative,
  this implies that $\deg(\phi_{j,i}) \leq N$ for $1\leq i,j\leq r$.
\end{proof}
  \subsection{Isogeny Diamonds and Kani's Lemma}
  We now give the definition of an isogeny diamond in the setting of
  abelian varieties. This was first introduced by Kani~\cite{Kani97}
  for elliptic curves and generalized in~\cite{Rob2023} to principally
  polarized abelian varieties.
  \begin{defn}
    A {\em $(d_1,d_2)$-isogeny factorization configuration} is a
    $d_1\cdot d_2$-isogeny $f: A \to B$ between principally polarized
    abelian varieties of dimension $g$ which has two factorizations
    $f = f_1' \circ f_1 = f_2' \circ f_2$ with $f_1$ a $d_1$-isogeny,
    $f_2$ a $d_2$-isogeny. If, in addition, $d_1$ and $d_2$ are
    relatively prime we call this configuration a {\em
      $(d_1,d_2)$-isogeny diamond configuration}.

    \begin{center}
\
\UseComputerModernTips
\xymatrix @-1pc
{
A \ar[rr]_{f_1}\ar[dd]_{f_2}&&A_1 \ar[dd]_{f_1'}\\
&&\\
A_2 \ar[rr]_{f_2'}&& B
}
\end{center}
\end{defn}
\begin{lem}[Kani's Lemma]
  Let $f =    f_1' \circ f_1 = f_2' \circ f_2$ be a
  $(d_1,d_2)$-isogeny diamond configuration. Then
  $F=\left(\begin{matrix}f_1& \tilde f_1'\\
      -f_2&f_2'\\
  \end{matrix}\right)$ is $d$-isogeny $F: A \times B \to A_1 \times
A_2$ with $d=d_1+d_2$ and kernel $\Ker F = \{(\tilde{f_1}(P),
f_1'(P)): P \in A_1[d]\}$.
\end{lem}
\begin{proof}
  This is Lemma~6 in~\cite{Rob2023}, which generalizes Theorem 2.3 in~\cite{Kani97}.
  \end{proof}

\subsection{Divide Algorithm}

The following gives more details for the algorithm in Proposition~\ref{prop:eta}, which is the algorithm described in \cite[Section 4]{Rob23b}. This algorithm also appears in~\cite[Section 4]{HW25}. 

\begin{defn}
Let $n$ be an integer whose factorization into prime powers is
$n=\ell_1^{e_1}\dots \ell_r^{e_r}$. We say that an integer $B$ is a
{\em powersmoothness bound} on $n$ and that $n$ is {\em $B$-powersmooth} if
\[ B \geq \max_{i}{\ell_i^{e_i}}.\]  
  \end{defn}

In the following algorithm, we are using endomorphisms given in HD representation which can be evaluated efficiently at $O(\log(\deg(\beta)))$-powersmooth torsion points of $E$.

\begin{alg}{Divide Algorithm}\label{alg:testing}

\textbf{Input:} A supersingular elliptic curve $E$ defined over $\FF_{p^k}$; $\beta \in \End(E)$ which is written as a sum $\beta = b_1\beta_1 + b_2 \beta_2 + b_3 \beta_3 + b_4 \beta_4$ where $\beta_i$ are linearly independent endomorphisms which can be evaluated efficiently at powersmooth torsion points of $E$ and $b_i \in \ZZ$; $n$ a positive integer; $Q$  the norm form such that $Q(x_1, x_2, x_3, x_4) = \deg(\sum_{i=1}^4 x_i \beta_i)$

\textbf{Output:} TRUE if $\frac{\beta}{n}$ is an endomorphism of $E$ and FALSE if $\frac{\beta}{n}$ is not an endomorphism.

\begin{enumerate}
    \item Compute $\deg(\beta)$. If $n^2 \nmid \deg(\beta)$, conclude that $\frac{\beta}{n}$ is not an endomorphism and output FALSE. Otherwise, set $N:=\deg(\beta)/n^2$.  
    \item Choose $a \in \ZZ$ such that $N':=N+a$ is powersmooth and $\gcd(N', pNn) = 1$.
    \item Compute integers $a_1, a_2, a_3, a_4 \in \ZZ$ such that
      $a_1^2 + a_2^2 + a_3^2 + a_4^2 = a$.
      Let $\alpha \in \End(E^4)$ be the $a$-isogeny given by the matrix \[\begin{pmatrix} a_1 & -a_2 & -a_3 & -a_4\\
    a_2 & a_1 & a_4 & -a_3\\
    a_3 & -a_4 & a_1 & a_2\\
    a_4 & a_3 & -a_2 & a_1\\        
    \end{pmatrix}.\]
    \item Compute $K:=\{(\frac{\widehat{\beta}}{n} \cdot \text{Id}_4(P), \alpha(P)): P \in E^4[N+a]\}.$ Note that $K$ can be computed even if $\frac{\beta}{n}$ is not an endomorphism: we can compute $\widehat{\beta}$ on $E[N+a]$, and by choice of $a$, $n$ is invertible mod $N+a$. 
    \item Determine if $F: E^8 \to E^8/K$ is an endomorphism of principally polarized abelian varieties. (We do so by computing an appropriate theta structure for $E^8/K$ and checking that the projective theta constant of $E^8$ is the same as the projective theta constant of $E^8/K$.) If not, then terminate and conclude that $\frac{\beta}{n}$ is not an endomorphism.
    \item Choose $M > \sqrt{\deg(\beta)} + \sqrt{n^2(N+a)}$ which is
      powersmooth. We check if
      $F_{ij}|_{E[M]} = \psi \frac{\beta}{n}|_{E[M]}$ for some
      $\psi \in \Aut(E),$ by evaluating the composition
      $E \xrightarrow{\iota_i} E^8 \xrightarrow{F} E^8
      \xrightarrow{\pi_j} E$ on $E[M].$ If for some $F_{ij}$ we have
      $F_{ij}|_{E[M]} = \psi \frac{\beta}{n}|_{E[M]}$, then we
      terminate and output TRUE. If no
      entry $F_{ij}$ satisfies $F_{ij} = \psi \frac{\beta}{n}$, then
      terminate and output FALSE.
   \end{enumerate}
\end{alg}

\begin{prop}\label{endoalg}
    Algorithm ~\ref{alg:testing} is correct and runs in time polynomial in $\log(p^k)$ and $\log(\deg(\beta))$.
\end{prop}

The proof of Proposition \ref{endoalg} follows from
Lemmas~\ref{lem:true}, \ref{lem:false}, and \ref{lem:complexity} below.

\begin{lem}\label{autppav} Let $\psi \in \Aut(E^n, \lambda)$,
  with $\lambda$ the product polarization. Suppose $\psi$ is given by its
  matrix form $M =(\psi_{i,j})_{1\leq i,j\leq n}$ as in
  Section~\ref{sec:notation}.
Then $M$ has exactly one nonzero entry in each row and each column.
  Whenever $\psi_{ij}$ is nonzero,
  $\psi_{ij}$ is an automorphism of $E$.
\end{lem}

\begin{proof}
  As $\psi$ preserves the polarization $\lambda$ on $E^n$,
  $\lambda = \psi^{\vee} \lambda \psi$. Therefore
  $\psi^{\dagger}\psi = 1$, with $\psi^{\dagger}$ the image
  of $\psi$ under the Rosati involution.
By \cite[Lemma 3]{Rob2023}, the matrix form of
$\psi^{\dagger}$ is given by the matrix $(\psi_{ij}^{\dagger})_{i,j}$, with $\psi^{\dagger}_{i,j}$ the
Rosati involution of $\psi_{i,j}$, which for elliptic curves equals
the dual isogeny $\widehat{\psi_{i,j}}$. Call this matrix $\hat{M}$.
Since $\psi^{\dagger}\psi = 1$, it follows that $\hat{ M} M = \text{Id}_n$.

Fix $1 \leq i \leq n$. We have $\sum_{k=1}^n
\widehat{\psi_{ik}}\psi_{ik} = \sum_{k=1}^n \deg(\psi_{ik}) = 1$. As
$\deg(\psi_{ik})$ is a positive integer whenever $\psi_{ik}$ is
nonzero, $\deg(\psi_{ik}) $ is nonzero for exactly one $k$, and for this $k$, $\deg(\psi_{ik}) = 1$. 

For $j \neq i$, we have $\sum_{k=1}^n \widehat{\psi_{ik}}\psi_{jk} =
0$. By the above argument, $\psi_{ik} = 0$ for all but one $k$. For
this $k$, the fact that
$\widehat{\psi_{ik}} \psi_{jk} = 0$ implies that $\psi_{jk} = 0$. 

This shows that there is a unique nonzero entry in the $i$-th row, and that it is the only nonzero entry in its column. As there are $n$ rows and $n$ columns, this shows that there is a unique nonzero entry in each column, which is necessarily an automorphism.
\end{proof}

\begin{lem} \label{lem:true}Let $\beta\in \End(E)$ and $n$ a positive integer. If $\frac{\beta}{n}$ is an endomorphism, then Algorithm ~\ref{alg:testing} outputs True.

\end{lem}

\begin{proof}
    Let $\phi = \frac{\beta}{n} \in \End(E)$. Then $\deg(\phi) = \frac{\deg(\beta)}{n^2} = N$. Since $\alpha$ is built out of scalar multiplications, we have the following commutative diagram, which is an $(N,a)$-isogeny diamond configuration.

    \[\begin{tikzcd}
    E^4 \arrow{r}{\phi \cdot \text{Id}_4} \arrow[swap]{d}{\alpha} & E^4 \arrow{d}{\alpha} \\
    E^4 \arrow{r}{\phi \cdot \text{Id}_4} & E^4
    \end{tikzcd}
    \]

    By Kani's Lemma, there is an (N+a)-endomorphism $G: (E^8, \lambda) \to (E^8, \lambda)$, where $\lambda$ is the product polarization, such that $G$ is given by the matrix $\begin{pmatrix}
        \phi \cdot \text{Id}_4 & \alpha^{\dagger}\\
        -\alpha & \widehat{\phi} \cdot \text{Id}_4\\
    \end{pmatrix}.$ Moreover, as $a$ was chosen such that $(N,a) = 1$, we can write $\ker(G) = \{\widehat{\phi} \cdot \text{Id}_4 (P), \alpha(P): P \in E^4[N+a]\},$ which is the subgroup $K$ constructed in Step 4.
    
If $F$ is an isogeny with $\ker(F) = K$, then $F$ is an $(N+a)$-endomorphism of principally polarized abelian varieties and the computed theta constants are equal. Therefore, we proceed to Step 6. 
    
    By~\cite[Proposition 1.1]{Kani97}, there is an automorphism $\psi: E^8 \to E^8$ which preserves the product polarization and such that $F = \psi G$. By Lemma~\ref{autppav} each row and each column of the matrix form of $\psi$ has exactly one nonzero entry, which is an automorphism of $E$. Thus, the entries of the matrix form of $F$ are precisely the entries of the matrix form of $G$, composed with an automorphism of $E$. In particular, four of the nonzero entries of $F$ will be given by $\psi_{ij} \phi$ for some automorphism $\psi_{ij} \in \End(E)$. 
\end{proof}

\begin{lem}
  The subgroup $K$ in Step 4 of Algorithm~\ref{alg:testing} is a
  maximally isotropic subgroup of $E^8[N+a]$ (whether or not
  $\frac{\beta}{n}$ is an endomorphism). Thus, $K$ is the kernel of an
  $(N+a)$-isogeny with respect to some polarization on $E^8$.
\end{lem}

\begin{proof}
  Let $K$ denote the subgroup in Step 4 of
  Algorithm~\ref{alg:testing}, which is precisely the image of
  $F^{\dagger} = \begin{pmatrix}
    \frac{1}{n}\widehat{\beta} \cdot \Id_4 & -\alpha^{\dagger}\\
    \alpha & \frac{1}{n}\beta \cdot \Id_4\\
    \end{pmatrix}$ on $(E^4 \times E^4)[N+a]$

    Let $m \in \ZZ$ such that $mn \equiv 1 \pmod{N+a}$. Consider the following isogeny factorization configuration:

    \[\begin{tikzcd}
    E^4 \arrow{r}{m\beta \cdot \text{Id}_4} \arrow[swap]{d}{mn\alpha} & E^4 \arrow{d}{mn\alpha} \\
    E^4 \arrow{r}{m\beta \cdot \text{Id}_4} & E^4
    \end{tikzcd}
    \]

    By Kani's Lemma, there is an $m^2n^2(N +a)$-endomorphism of $E^8$
    with respect to the product polarization, given by
    $F' = \begin{pmatrix}
      m\beta \cdot \Id_4 & mn\alpha^{\dagger}\\
      -mn\alpha & m\widehat{\beta}\cdot \Id_4\\
    \end{pmatrix}$ and with kernel equal to the image of $F'^{\dagger} = \begin{pmatrix}
        m\widehat{\beta}\cdot \Id_4 & -mn\alpha^{\dagger}\\
        mn\alpha & m\beta \cdot \Id_4
    \end{pmatrix}$ on $(E^{4}\times E^4)[m^2n^2(N+a)].$ Let $K' = F'^{\dagger}(E^4 \times E^4)[m^2n^2(N+a)].$ By Kani's Lemma, $K'$ is a maximal isotropic subgroup of $E^8[m^2n^2(N+a)]$. 

    First, $K' \cap E^{8}[N+a]$ is a maximal isotropic subgroup of
    $E^8[N+a].$ Let $e_{m^2n^2(N+a)}$ be the Weil pairing on
    $E^8[m^2n^2(N+a)]$ and $P,Q \in E^{8}[N+a] \cap K'$.
By compatibility of
    the Weil pairing,
    $1 = e_{m^2n^2(N+a)}(P,Q) = e_{N+a}(mnP, mnQ)$. By choice of $m$, we have
    $e_{N+a}(mnP, mnQ) = e_{N+a}(P,Q).$ Thus, $K' \cap E^8[N+a]$ is an
    isotropic subgroup of $E^8[N+a]$.  Since $K'$ is a maximal
    isotropic subgroup of $E^{8}[m^2n^2(N+a)]$, and $(m^2n^2, N+a)=1$,
    we have $K' \cap E^8[N+a]$ has order $(N+a)^8$ and is therefore a
    maximal isotropic subgroup of $E^8[N+a]$.

    Finally, we have $K = K' \cap E^{8}[N+a]$. It is clear that $K \subset K' \cap E^{8}[N+a]$, since $F^{\dagger} = F'^{\dagger}$ on $E^{8}[N+a]$. Moreover, by the description of $K$ as $\{(\frac{\widehat{\beta}}{n} \cdot \text{Id}_4(P), \alpha(P)): P \in E^4[N+a]\}$, where $\beta$ and $\alpha$ have degrees coprime to $N+a$, it is clear that the order of $\#K = (N+a)^8 = \#(K' \cap E^8[N+a]).$ Thus, $K$ is a maximal isotropic subgroup of $E^8[N+a]$.

By~\cite[Proposition 1.1]{Kani97}, $K$ is therefore the kernel of an $N+a$-isogeny with respect to some polarization.
    \end{proof}

The following lemma shows that an endomorphism is uniquely determined by its degree and its action on $M$-torsion, for suitably large $M$ (depending on the degree).

\begin{lem}\label{lem:equality} Let $E$ be an elliptic curve and $\phi, \psi \in \End(E).$ Let $M > \sqrt{\deg(\phi)} + \sqrt{\deg(\psi)}.$ If $\psi|_{E[M]} = \phi|_{E[M]}$, then $\psi = \phi$.
\end{lem}

\begin{proof} For contradiction, assume the hypotheses of the lemma
  and that $\phi - \psi$ is nonzero.  Since
  $\psi|_{E[M]} = \phi|_{E[M]}$, $E[M] \subset \ker(\phi - \psi).$
  Since $\phi - \psi$ is nonzero, we must have $\phi - \psi= M \gamma$
  for some nonzero $\gamma \in \End(E)$. Thus,
  $\deg(\phi - \psi) = M^2 \deg(\gamma).$
By the Cauchy-Schwartz inequality, $\deg(\phi - \psi) \leq (\sqrt{\deg(\phi)} + \sqrt{\deg(\psi)})^2$.
Hence $M^2 \leq M^2 \deg(\gamma) \leq (\sqrt{\deg(\phi)} +
\sqrt{\deg(\psi)})^2$, which is a contradiction.
\end{proof}

\begin{lem} \label{lem:false}If $\frac{\beta}{n}$ is not an endomorphism, Algorithm ~\ref{alg:testing} outputs False.

\end{lem}

\begin{proof}
  Assume $F: E^8 \to E^8$ respects the product polarization and has
  kernel $K$ as defined in Step 4. Let $F_{ij}$ be an entry in the
  matrix form of $F$. Then $\deg(F_{ij}) \leq (N + a)$. If
  $F_{ij}|_{E[M]} = \frac{\psi\beta}{n}|_{E[M]}$ for some
  $M > \sqrt{\deg(\beta)} + \sqrt{n^2(N+a)}$ and an automorphism
  $\psi$, then $n F_{ij}|_{E[M]} = \psi \beta|_{E[M]}$. As we know
  $\psi \beta, n F_{ij}$ are endomorphisms, and
  $M > \sqrt{\deg(\beta)} + \sqrt{n^2(N+a)} > \sqrt{\deg(\psi \beta)}
  + \sqrt{n \deg(F_{ij})}$, Lemma~\ref{lem:equality} implies that
  $\frac{\beta}{n} = \psi^{-1} F_{ij} \in \End(E)$.
    \end{proof}

\begin{lem} \label{lem:complexity} Algorithm~\ref{alg:testing} runs in time polynomial in $\log(p^k)$ and $\log(\deg(\beta))$.
\end{lem}

\begin{proof}
  Let $B$ be a powersmoothness bound for $N+a$ (as in Step 2), and let
  $C$ be a powersmoothness bound for $M$ (as in Step 6). Given $Q$,
  computing the degree $\deg(\beta)$ amounts to evaluating $Q$ at
  $(b_1, b_2, b_3, b_4)$. The complexity of computing
  $a_1, a_2, a_3, a_4$ is
  $O((\log(a))^2 (\log\log(a))^{-1})$, see \cite{RS86, PT18}.

Computing a basis for $K$ means first computing a basis for $E[N+a]$; decomposing into at most $\log(N+a)$ prime power parts, this can be done in $O(B^2 \log(p^k)^2\log(N+a))$ operations \cite[Lemma 7]{Rob2023}. Evaluating $\widehat \beta$ on a basis for $E[N+a]$ and $\alpha$ on the induced basis for $E^4[N+a]$ can be done efficiently by our assumption on $\beta$ and powersmoothness of $N+a$.

For Step 5, we need to check that $F$ is truly an endomorphism. We
place the additional data of a symmetric theta structure of level 2 on
$E^8$, by taking an appropriate symplectic basis of $E[4]$ if $N+a$ is
odd, or $E[2^{m+2}]$ where $2^{m}$ is the largest power of 2 dividing
$N+a$ otherwise. (See Proposition C.2.6 of~\cite{DLRW23} and the
preceding remark about how to choose a basis which is compatible with
$K$ in different cases.) Decomposing $K$ into prime components, we can compute the theta null point of
$E^8/K$ with the induced theta structure in
$O(\ell_{N+a}^{8}\log(N+a))$ operations, where $\ell_{N+a}$ is the
largest prime dividing $N+a$. (See Theorem C.2.2 and Theorem C.2.5 of
~\cite{DLRW23}.) Finally, as $F$ may not preserve the product theta
structure even if it is the desired endomorphism, we need to act on
the theta null point by a polarization-preserving matrix in order to
directly compare theta null points. When $N+a$ is odd, this matrix is
computed explicitly ~\cite[Proposition C.2.4]{DLRW23} from the action
of $F$ on $E[4]$, which can also be evaluated in
$O(\ell_{N+a}^8 \log(N+a))$ operations. This gives $O(B^8 \log(N+a))$
operations for this step.

In Step 6, computing a basis for the prime-power parts of $E[M]$ takes $O(C^2 \log(p^k)^2\log(M))$ operations. If $F$ is an endomorphism, then having already computed theta coordinates for $E^8$ and $E^8/K$ in the previous step, we can evaluate $F$ in terms of theta coordinates ~\cite[Theorem C.2.2, Theorem C.2.5]{DLRW23} and translate back to Weierstrass coordinates to check the equality. Note that there are only finitely many, and usually two, automorphisms to consider. Each evaluation costs $O(\ell_{N+a}^8 \log(N+a))$ operations where $\ell_{N+a}$ is the largest prime dividing $N+a$. There are 64 entries $F_{ij}$ to check, by checking the equality on at most $2 \log(M)$ points. Thus, this step requires at most $O(C^2 \log(p^k)^2 \log(M) + B^8 \log(N+a) \log(M))$ operations.

Now, we show that $B$ and $C$ can be taken of size $O(\log(\deg(\beta)))$ and that $N+a$ and $M$ are polynomial in $\deg(\beta)$.

Let $\tilde{M}:= \sqrt{\deg(\beta)} +\sqrt{n^2(N+a)}$. By~\cite{RS62},
Theorem 4 and its Corollary, there exist constants $c,c'$ such that
\[
  k<\log\prod_{p_i \leq ck} p_i < c\cdot c' k.
\]

Multiply successive primes until their product is bigger than $\tilde
M$.
By the above relation, it is enough to choose the 
primes $p_i \leq c \cdot \log(\tilde M)$ primes, then 
$\log{\tilde M}< \log\prod_{p_i \leq ck} p_i$, or $\prod_{p_i \leq  ck}p_i>\tilde M$. 
Let $M:=\prod_{p_i  \leq  ck}p_i$ with $k=\log(\tilde M)$. Then
\[
  \log M = \log\prod_{p_i \leq c \log(\tilde M)} p_i < c\cdot c' \log \tilde M,
  \]
so $\log M = O(\log(\tilde M))$, and by construction, $M$ is $O(\log \tilde M)$-smooth.

Multiply successive primes which are coprime to $pNn$ until their product is bigger than $N$.   
If $p > c \log(N^2 n)$, then it is enough to choose primes $p_i \leq c \cdot \log(N^2 n)$ which are coprime to $pNn$. 
We have
\[
  \log(N^2 n) <\log\prod_{p_i \leq c\log(N^2 n)} p_i,
\] so $N^2 n < \prod_{p_i \leq c \cdot \log(N^2n)} p_i$.
Since $(\prod_{p_i \mid Nn} p_i) \leq Nn$, we get
\[\prod_{p_i \leq c \cdot \log(N^2n), (p_i, Nn) = 1} p_i \geq \frac{N^2 n}{N n} = N\]
and $N+a$ is $O(\log(N^2 n))$-powersmooth. 
Since $Nn^2 = \deg(\beta)$, this shows $N+a$ is $O(\log(\deg(\beta)))$-powersmooth.

A similar argument works for $p < c\log(N^2n)$; in cryptographic applications, this case will never occur.

Now we bound $N+a$ in terms of $\deg(\beta)$. 
\[
 \log(N+a) < \log\prod_{p_i \leq c \cdot \log(N^2n)} p_i < c' c \cdot \log(N^2 n) \leq c' c \log(\deg(\beta)^2),
  \]
so \[N+a < e^{cc'\log((\deg(\beta))^2)} = \deg(\beta)^{2cc'}.
\]

Then a bound for $\tilde{M}$ is given by $\tilde{M} \leq \sqrt{\deg(\beta)} + \sqrt{\deg(\beta)(\deg(\beta))^{2cc'}}$. This shows that $M$, which is $O(\log(\tilde M))$-smooth, is $O(\log(\deg(\beta)))$-smooth.
\end{proof}

One can get speedups by replacing $E^8$ by $E^4$ and tweaking parameters as discussed by Robert in~\cite[Section 6]{Rob2023}; for simplicity and for a proven complexity we don't go into those details here. 

\section{An Explicit Isomorphism with the Matrix Ring}               
In this section, we give more details of Propositions~\ref{Oq} and ~\ref{iso}, which are consequences of~\cite{Voight13}. 

First, we restate and prove Proposition~\ref{Oq}.

\begin{prop}\label{app:Oq} Suppose an order $\mathcal{O}_0 \subset \End(E)$ is given by a basis, a multiplication table, and $Q$ the norm form on $\OO_0$, and let $q$ be a prime. Then there is an algorithm which computes a $q$-maximal $q$-enlargement $\mathcal{O}_q$ of $\mathcal{O}_0$.
The run time is polynomial in the size of the basis and multiplication table. The basis elements which are output are of the form $\frac{\beta}{q^k}$, for an endomorphism $\beta \in \mathcal{O}_0$ and $k \leq e = v_q(\discrd(\mathcal{O}_0))$. Furthermore, $\deg(\beta)$ is polynomial in the degrees of basis elements of $\mathcal{O}_0$.\end{prop}

\begin{proof}  On input $\mathcal{O}_0$, specified by the multiplication table and
  $Q$, we compute a $q$-maximal $q$-enlargement of $\mathcal{O}_0$, denoted $\OO_q$ 
  \cite[Algorithms 3.12, 7.9, 7.10]{Voight13}. More specifically,
  Algorithm 3.12 produces a basis for $\mathcal{O}_0 \otimes \ZZ_q$
  such that the norm form is normalized. Algorithm 7.9 gives a basis
  for a potentially larger ``$q$-saturated" order, whose elements are
  of the form $\frac{x}{q^k}$. Here, $x$ has coefficients in terms of
  the original basis at most
  $\max (\Trd(\beta_i \hat{\beta}_j))^4$, where $\beta_i$ and
  $\beta_j$ range over basis elements of the original basis. The power
  $k$ in the denominator is at most $\lfloor j/2 \rfloor$ where $j$ is
  the valuation of the atomic form corresponding to the basis element,
  and hence $k \leq e = v_q(\discrd(\mathcal{O}_0))$.

Since $|\Trd(\beta_i \hat\beta_j)| \leq 2\sqrt{\deg(\beta_i) \deg(\beta_j)}$, the coefficients are polynomial in the original basis degrees. Applying Algorithm 7.10 of~\cite{Voight13} adjoins a zero divisor mod $q$, which is of the form $\frac{x}{q}$; here, $x$ is expressed as linear combinations of the original basis with polynomially-sized coefficients. Thus, the basis which is output for $\OO_q$ has coefficients which are polynomially-sized in the degrees of the original basis elements and $q$. Therefore, a basis element $\frac{\beta}{q^k}$ satisfies $\deg(\beta)$ is at most polynomially-sized in $\deg(\beta_i)$, where $\beta_i$ ranges over the original basis elements, and $q$.
\end{proof}

We break up the proof of Proposition~\ref{iso} into two steps. First, we compute a zero divisor $x$ mod $q^{r+1}$.

\begin{prop}\label{zerodiv} Given a basis and multiplication table for a $q$-maximal order $\OO_q$, and an integer $r$, there is an algorithm which computes a zero divisor $x \in \OO_q \otimes \ZZ_q$ mod $q^{r+1}$. In other words, there is an algorithm to compute an element $x \in \OO_q \otimes \ZZ_{(q)}$ such that there exists a zero divisor $x' \in \OO_q \otimes \ZZ_q$ with $v_q(x - x') \geq r+1$. The element $x$ is expressed as a linear combination of the given basis such that coefficients are polynomially-sized in $q^{r+1}$ and $\deg(\beta_i)\deg(\beta_j)$, where $\beta_i$ and $\beta_j$ range over elements of the given basis. The runtime is polynomial in $\log(q^{r+1})$ and the size of $\OO_q$.\end{prop}

\begin{proof} 
First, use~\cite[Algorithm 3.12]{Voight13} on $\OO_q
\otimes \ZZ_q$ to obtain a normalized basis $\{f_1, f_2, f_3, f_4\}$
for $\OO_q \otimes \ZZ_q$. By clearing denominators by units in $\ZZ_q$ if necessary, we can ensure $f_i \in \mathcal{O}_0 \otimes \ZZ_{(q)}$.

As $\OO_q$ is $q$-maximal, the output basis being normalized means that the reduced norm form $Q(x_1, x_2, x_3, x_4) = \Nrd(\sum_{i=1}^4 x_i f_i)$ is given by a sum of atomic forms. 

When $q$ is odd, this means that $Q(x_1, x_2, x_3, x_4) = \sum_{i=1}^4 a_i x_i^2$ where $a_i \in (\ZZ_q)^\times$ and $\Trd(f_i \hat{f_j}) = 0$ when $i \neq j$. When $q=2$, atomic forms are of one of the two following types: (i) $ax^2$ for $a \in (\ZZ_q)^{\times}$ or (ii) $a_ix_i^{2} + a_{ij}x_ix_j + a_{j}x_i^2$ such that $v_2(a_{ij}) \leq v_2(a_i) \leq v_2(a_j)$ and $v_2(a_i)v_2(a_{ij}) = 0$. Up to reordering basis elements if necessary, we may therefore write $Q(x_1, x_2, x_3, x_4) = A_{12}(x_1, x_2) + A_{34}(x_3, x_4)$, where $A_{ij}$ is either atomic of type (ii) or a sum of atomic forms of type (i).

We split up rest of the proof into the case that $q$ is odd and $q=2$: We first produce a nonzero element $x \in (\ZZ/q\ZZ)^4$ such that $Q(x) \equiv 0 \pmod{q}$. Then, we show that there exists a lift $x'$ in $\OO_q \otimes \ZZ_q$, and we compute and output a lift of $x$ in $\OO_q \otimes \QQ$ up to our desired precision $q^r$. In each case, the coefficients (in terms of the $f_i$) $x_1, x_2, x_3, x_4$ will be chosen mod $q^r$, so the resulting output coefficients (in terms of the input basis) is polynomially-sized in $q^r$ and $\deg(\beta_i)\deg(\beta_j)$.

\textbf{Case 1:} $q$ is odd.   In this case, the resulting reduced norm form is $Q(x_1, x_2, x_3, x_4) = \Nrd(\sum_{i=1}^{4} x_i f_i) = \sum_{i=1}^4 a_ix_i^2$. The coefficients $a_i$ may be rational, but $v_q(a_i) = 0$, so we may replace $a_i$ by an integer mod $q^r$. Then there is a nonzero solution $(x_1, x_2, x_3) \in (\FF_q)^3$ to the equation $\sum_{i=1}^3 a_i x_i^2 \equiv 0$, which can be found by a deterministic algorithm running in polynomial time in $\log(q)$ \cite{W5}. Reindexing the basis elements $f_i$ and the corresponding $a_i$ as necessary, we can assume $x_1 \neq 0$, so that the quadratic polynomial $Q_1(x) = Q(x, x_2, x_3, 0)$ has a nonzero solution, $x_1$, mod $q$. Furthermore, $Q_1'(x_1) = 2a_1x_1$, which is nonzero mod $q$. Thus, by Hensel's Lemma, $x$ can be lifted to a solution to $Q_1(x) = 0$ over $\ZZ_q$. A solution mod $q^{r+1}$ can be recovered in (at most) $r$ Hensel lifts, each running in polynomial time in $\log(q)$ (see \cite[Algorithm 15.10 and Theorem 15.11]{vzGG13} or \cite[Theorem 3.5.3 ]{Coh93}).

\textbf{Case 2:} $q = 2$.  In this case, the resulting reduced norm form is given by the normalized form $Q(x_1, x_2, x_3, x_4) = A_{1,2}(x_1, x_2) + A_{3,4}(x_3, x_4)$. Here $A_{i,j}(x_i, x_j) = a_{i}x_i^2 + a_{i,j}x_ix_j + a_jx_j^2$. The discriminant of $Q$, and therefore of $\OO_q \otimes \ZZ_q$, is $(4a_1a_2 - a_{1,2}^2)(4a_3a_4 - a_{ij})^2$. As $\OO_q \otimes \ZZ_q$ is $2$-maximal, $a_{i,2}$ and $a_{3,4}$ are necessarily nonzero (mod 2).

Let $A(y,z)$ be an atomic form of type (ii), say $A(y,z) = ay^2 + byz + cz^2$ such that $v_2(b) \leq v_2(a) \leq v_2(c)$. Further assume $v_2(b) = 0$. We show that we can choose $y_0,z_0 \in \ZZ/2\ZZ$ such that $A(y_0, z_0) \equiv 1 \pmod{2}$ and at least one of $y_0$ or $z_0$ is odd. If $v_2(a) \geq 1$ (and therefore $v_2(c) \geq 1$ as well), or if $v_2(a) = v_2(c) = 0$, we can set $y_0 \equiv z_0 \equiv 1 \pmod{2}$. Otherwise, in the case that $v_2(a) = 0$ and $v_2(c) > 0$, we can set $y_0 \equiv 1 \pmod{2}$ and $z_0 \equiv 0 \pmod{2}$.

The quadratic form $Q(x_1, x_2, x_3, x_4)$ is the sum of two atomic quadratic forms $A_{1,2}$ and $A_{3,4}$ as above. We obtain a solution mod $2$ by choosing $x_1, x_2, x_3, x_4$ mod 2 as just described. If $x_1$ and $x_2$ are both odd, i.e. in the case that $a_{1}$ and $a_2$ are of the same parity, we lift  $x_2, x_3, x_4$ to $\ZZ/q^r\ZZ$ to obtain a quadratic polynomial $Q_1(x) = Q(x, x_2, x_3, x_4)$ with a solution mod 2 at $x \equiv 1 \pmod{2}$. Then the derivative $Q_1'(1)= 2a_1 + a_{1,2}x_2$ is a unit in $\ZZ_2$. Otherwise, in the case that $x_1$ is odd and $x_2$ is even, we fix integers $x_1, x_3, x_4 \in \ZZ/q^r\ZZ$ to obtain a quadratic polynomial $Q_2(x) = Q(x_1, x, x_3, x_4)$ with a solution mod 2 at $x \equiv 0 \pmod{2}$. Then the derivative $Q_2'(0) = a_{1,2}x_1$ is a unit in $\ZZ_2$. In either case, we obtain a solution to $Q=0$ mod $2$ which can be lifted to a solution in $\ZZ_q^4$ via Hensel's Lemma. As in the case that $q$ is odd, a solution mod $q^{r+1}$ can be recovered in $r$ lifts, running in polynomial time in $\log(q)$. 
\end{proof}

We restate and prove Proposition~\ref{iso}.

\begin{prop} 
Let $q\neq p$. Given a $q$-maximal order $\OO_q \subset \End(E) \otimes \QQ$ and a nonnegative integer $r$, there is an algorithm which computes an isomorphism $f: \OO_q \otimes \ZZ_q \to M_2(\ZZ_q)$ modulo $q^{r+1}$. This isomorphism is specified by giving the inverse image of standard basis elements $i'$ and $j'$ determined mod $q^{r+1}$ in $\OO_q$, such that \[j' \mapsto \begin{pmatrix} 0 & 1\\1 & 0 \end{pmatrix}\] and \[i' \mapsto \begin{pmatrix} 1 & 0\\ 0 & -1 \end{pmatrix} \text{ if } q \neq 2, \begin{pmatrix}0 & 1 \\ 1 & 1 \end{pmatrix} \text{ otherwise}.\] 
The run time is polynomial in $\log(q^r)$ and the size of the basis and multiplication table for $\OO_q$. 
In terms of the basis for $\OO_q$, the representatives $i'$ and $j'$ are expressed with coefficients which are determined mod $q^{r+1}$
\end{prop}

\begin{proof}

By Proposition~\ref{zerodiv}, there is an algorithm to compute $x \in \OO_q \otimes \ZZ_{(q)}$ such that $\Nrd(x) \equiv 0 \pmod{q^{r+1}}$. We first use $x$ as input for \cite[Algorithm 4.2]{Voight13}to compute nonzero $e \in \OO_q \otimes \ZZ_q$ such that $e^2 = 0$. As in the proof of Proposition~\ref{zerodiv}, we only specify $e$ up to precision $q^{r+1}$ and can therefore approximate $e$ with an element of $\OO_q \otimes \ZZ$. Furthermore, we can choose $e = \sum_{i=1}^4 e_i f_i$ such that for some $i$, $q\nmid e_i$.

Then, on input $e$, we use \cite[Algorithm 4.3]{Voight13} to compute $i'$ and $j'$ as a $\ZZ$-linear combination of $\frac{1}{s}e$ and $\frac{1}{s}f_ie$, for a basis element $f_i$ such that $s=\Trd(f_ie)$ is nonzero. 

In fact, we will modify the algorithm by choosing $f_i$ such that $\Trd(f_ie)$ is nonzero mod $q$. If no such $i$ exists, then $\Trd(ye) = 0$ for all $y \in \OO_q$, so we show this cannot happen. Write $y = \sum_{j=1}^4 y_j f_j$ and $e = \sum_{i=1}^4 e_i f_i$, and consider the expression for $\Trd(ye) = - \Trd(y \bar{e})$ given by $\sum_{j=1}^4 \sum_{i=1}^4 -y_i e_j \Trd(f_i \hat{f_j})$. As $\{f_1, f_2, f_3, f_4\}$ is a normalized basis, the equation simplifies in the following ways, depending on if $q$ is even or odd.

If $q$ is odd, then the expression simplifies to $\sum_{i=1}^4 -e_i \Trd(f_i \hat{f_i}) y_i$. This is identically 0 mod $q$ if and only if $q$ divides $e_i \Trd(f_i \hat{f_i})$ for all $i$. In the notation of the proof of Proposition~\ref{zerodiv}, $\Trd(f_i \hat{f_i})$ is exactly $2a_i$ and hence is not divisible by $q$ by $q$-maximality. Hence, this expression is identically 0 mod $q$ if and only if $q$ divides $e_i$ for all $i$, and we chose $e$ such that this does not happen.

If $q=2$, we have that $\Trd(f_i \hat{f_i}) = 2 \Nrd(f_i) \equiv 0 \pmod{q}$ for all $i$, so the only nonzero terms are $-e_1 \Trd(f_2 \hat{f_1}), -e_2 \Trd(f_1 \hat{f_2}), -e_3 \Trd(f_4 \hat{f_3}), -e_4 \Trd(f_3 \hat{f_4}).$ We have $\Trd(f_1 \hat{f_2})=Trd(f_2 \hat{f_1})=a_{1,2}$ and $\Trd(f_3 \hat{f_4})=\Trd(f_4 \hat{f_3})=a_{3,4}$, which are not divisible by $q$ as we showed in the proof of Proposition~\ref{zerodiv}. Hence this expression is identically 0 mod $q$ if and only if $q$ divides $e_i$ for all $i$, but we chose $e$ such that this does not happen.

This shows that $v_q(\Trd(ef_i)) = 0$ for some $i$, so that $\frac{1}{s} \in \ZZ_q$, and the elements $i'$ and $j'$ output by Algorithm 4.3 of ~\cite{Voight13} (with this modification) are elements of $\OO_q \otimes \ZZ_q$ and furnish an isomorphism of $\OO_q \otimes \ZZ_q \to M_2(\ZZ_q)$. To get $i'$ and $j'$ in $\OO_q$ rather than in $\OO_q \otimes \QQ$, replace $\frac{1}{s}$ by an integer $m \equiv s^{-1} \pmod{q^r}$.
\end{proof}


\end{document}